\documentclass[12pt]{article}
\usepackage{amsmath,amssymb,amsthm,amsfonts,hhline,color}
\usepackage{euscript}
\usepackage{bbm}
\usepackage{hyperref}

\newcommand{\dfr}[2]{\dfrac{#1}{#2}}
\newcommand{\cd}{\cdot}

\newcommand{\dsum}{\displaystyle \sum}
\newcommand{\ol}[1]{\overline{#1}}

\renewcommand{\l}{\left}
\renewcommand{\r}{\right}

\newcommand{\vsv}{\vspace{5mm}}
\newcommand{\vsb}{\vspace{2mm}}
\newcommand{\q}{\quad}
\newcommand{\qq}{\qquad}

\newcommand{\la}{\langle}
\newcommand{\ra}{\rangle}

\newcommand{\abs}[1]{\l|{#1}\r|}

\newcommand{\Z}{\mathbb{Z}}
\newcommand{\C}{\mathbb{C}}

\newcommand{\F}{\mathbb{F}}

\newcommand{\vir}{\mathrm{Vir}}
\newcommand{\aut}{\mathrm{Aut}}
\newcommand{\wt}{\mathrm{wt}}

\renewcommand{\hom}{\mathrm{Hom}}

\newcommand{\id}{\mathrm{id}}
\newcommand{\om}{\omega}
\newcommand{\be}{\beta}
\newcommand{\al}{\alpha}

\newcommand{\w}{\omega}
\newcommand{\vacuum}{\mathbbm{1}}
\newcommand{\vac}{\vacuum}

\makeatletter
\@addtoreset{equation}{section}

\makeatother

\theoremstyle{plain}
\newtheorem{thm}{Theorem}[section]
\newtheorem{prop}[thm]{Proposition}
\newtheorem{lem}[thm]{Lemma}
\newtheorem{cor}[thm]{Corollary}

\theoremstyle{definition}
\newtheorem{df}[thm]{Definition}
\newtheorem{nota}[thm]{Notation}

\newtheorem{rem}[thm]{Remark}

\title{On $3$-transposition groups generated by $\sigma$-involutions associated to $c=4/5$ Virasoro vectors}
\author{
  Ching Hung Lam\footnote{Partially supported by NSC grant
  100-2628-M-001005-MY4 }%
  \vsb\\
  {\small \it Institute of Mathematics, Academia Sinica, Taipei 10617, Taiwan}\\
  {\small \it and }\\
  {\small \it National Center for Theoretical Sciences, Taiwan}\\
  {\small e-mail: {\tt chlam@math.sinica.edu.tw}}
  \vsv\\
  Hiroshi Yamauchi\footnote{Partially supported by JSPS Grant-in-Aid for Young 
  Scientists (B) No.~21740011 and No.~24740027.}
  \vsb\\
  {\small \it Department of Mathematics,
  Tokyo Woman's Christian University}\\
  {\small \it 2-6-1 Zempukuji, Suginami-ku, Tokyo 167-8585, Japan}\\
  {\small e-mail: {\tt yamauchi@lab.twcu.ac.jp}}
  \vsv
}
\date{}

\newcommand{\dih}[2]{\mathrm{DIH}_{#1}(#2)}

\newcommand{\sfr}[2]{\leavevmode\kern-.05em
  \raise.5ex\hbox{\the\scriptfont0 #1}\kern-.1em
  /\kern-.15em\lower.25ex\hbox{\the\scriptfont0 #2}\kern.02em}

\DeclareMathOperator*{\tensor}{\otimes}
\DeclareMathOperator*{\fusion}{\boxtimes}
\newcommand{\ann}{\mathrm{Ann}}
\newcommand{\Span}{\mathrm{Span}}
\newcommand{\au}{\nu}
\newcommand{\longto}{\longrightarrow}

\newcommand{\W}{\EuScript{W}}
\newcommand{\pii}{\pi \sqrt{-1}\, }
\newcommand{\Ker}{\mathrm{Ker}\,}
\newcommand{\rank}{\mathrm{rank}\,}
\newcommand{\cent}{\mathrm{Cent}}

\newcommand{\ee}{\mathbf{e}}
\newcommand{\AR}[1]{ \mathcal{A}_{#1}}

\begin{document}

\baselineskip 6mm

\maketitle

\begin{abstract}
In this paper, we show that $\sigma$-involutions associated to extendable
$c=4/5$ Virasoro vectors generate a $3$-transposition group in the automorphism group
of a vertex operator algebra (VOA).
Several explicit examples related to lattice VOA are also discussed in details.
In particular, we show that the automorphism group of the VOA
$V_{K_{12}}^{\hat{\au}}$ associated to the Coxeter Todd lattice $K_{12}$ contains
a subgroup isomorphic to $^+\Omega^{-}(8,3)$.
\end{abstract}


\tableofcontents

\renewcommand{\arraystretch}{1.5}

\section{Introduction}

Miyamoto's work on involutions associated to simple Virasoro vertex operator algebra $L(1/2,0)$ is a very beautiful theory in vertex operator algebra (VOA) theory.
Partially motivated by
the work of Conway\,\cite{C} and Dong et al.~\cite{DMZ}, Miyamoto
\cite{M1} discovered a simple method to construct an involutive automorphism associated
to a sub VOA isomorphic to the simple Virasoro vertex operator algebra $L(1/2,0)$ as follows:

Let $W\cong L(1/2,0)$ be a sub VOA of a VOA $V$ and let $e$ be the conformal element of $W$.
Let $V_e(h)$ be the sum of all irreducible $W$-submodules of $V$ isomorphic
to $L(1/2,h)$ for $h=0,1/2,1/16$.
Then one has the isotypical decomposition:
$$
  V=V_e(0)\oplus V_e(1/2)\oplus V_e(1/{16}).
$$
Define a linear automorphism $\tau_e$ on $V$ by
$$
  \tau_e=
  \begin{cases}
    \ \  1 & \text{ on }\  V_e(0)\oplus V_e(1/2),
    \vsb\\
    -1& \text{ on }\  V_e({1}/{16}).
\end{cases}
$$
Then $\tau_e$ becomes an automorphism on the VOA $V$.

When $V=V^\natural$ is the Moonshine VOA, then the automorphism $\tau_e$ defines a $2A$-involution in the Monster simple group. Miyamoto also showed in \cite{M4} (see also \cite{Ho}) that there is a one to
one correspondence between the $2A$ elements of the Monster and
sub VOAs of the Moonshine vertex operator algebra which is
isomorphic to the simple Virasoro VOA $L(\frac{1}{2},0)$. In fact,
Miyamoto gives a new construction of  the famous Moonshine VOA in \cite{M4} directly using
simple Virasoro VOA $L(\frac{1}{2},0)$.

On the fixed point subalgebra $V^{\tau_e}=V_e(0)\oplus V_e(1/2)$,
one can define another linear automorphism $\sigma_e$ by
$$
  \sigma_e=
  \begin{cases}
    \ \  1 & \text{ on }\  V_e(0),
    \vsb\\
    -1& \text{ on }\  V_e({1}/{2}).
  \end{cases}
$$
Then $\sigma_e$ defines an automorphism on $V^{\tau_e}$.
The Virasoro vector $e$ is said to be $\sigma$-type in $V$ if $\tau_e=id_V$.
In \cite{M1}, Miyamoto showed that if $E$ is a set of Virasoro vectors
of $\sigma$-type, then the automorphism subgroup $G=\la \sigma_e\mid
e\in E\ra$ is a 3-transposition group. Such kind of groups are
classified by Matsuo\,\cite{Ma2} and they are 3-transposition groups of
symplectic type \cite{CH}.

It is natural to ask if Miyamoto's theory can be generalized to other
simple Virasoro VOAs. In \cite{M2}, Miyamoto also showed that one can
associate an automorphism $\xi_U$ to a sub VOA $U$isomorphic to the VOA
$L(4/5,0)\oplus L(4/5,3)$, which describes the critical point of the
$3$-state Potts model (See also\cite{zf}), and $\xi_U^3=id_V$.  On the
fixed point subalgebra $V^{\xi_U}$, one can define an involutive
automorphism $\sigma_U$ on $V^{\xi_U}$ (see Proposition \ref{prop:2.8}).

Let $U\cong L(4/5,0)\oplus L(4/5,3)< V$ and let $u$ be the conformal vector of $U$. Then $u$ is said to be of {\it$\sigma$-type}
if $\xi_U=id_V$. In this article, we will  study groups generated by $\sigma$-involutions associated to extendable $c=4/5$ Virasoro
vectors. We will show that they generate a $3$-transposition group. Several explicit examples related to lattice VOAs will be
discussed in details. In particular, we show that some $3$-transposition group of orthogonal type can be realized in this manner.
One important example is the VOA $V_{K_{12}}^{\hat{\tau}}$ associated to the Coxeter-Todd lattice $K_{12}$. We show that
$\sigma$-involutions associated to extendable $c=4/5$ Virasoro vectors in $V_{K_{12}}^{\hat{\tau}}$ generates a subgroup
isomorphic to $^+\Omega^{-}(8,3)$, which we believe is related to one of the 3-local subgroups $3^{8}.\Omega^-(8,3)$ of the
Monster simple group.

The organization of this article is as follows. In Section 2, we recall some basic properties of unitary Virasoro VOAs and their
extensions. The representation theory of the VOA $\W(4/5)\cong L(4/5,0)\oplus L(4/5,3)$ will also be reviewed. In particular,
certain automorphisms associated to $\W(4/5)$ will be defined. In Section 3, we study $\sigma$-involutions associated to
extendable $c=4/5$ simple Virasoro vectors of $\sigma$-type in a VOA $V$. We show that they generate a $3$-transposition
group in $\aut(V)$. In Section 4, we discuss the Griess algebras generated two or three extendable $c=4/5$ simple Virasoro
vectors of $\sigma$-type.  Several explicit examples will be given in Section 5. The lattice type VOA $V_{K_{12}}^{\hat{\tau}}$
associated to the Coxeter-Todd lattice $K_{12}$ is of special interesting since it contains certain  exceptional Virasoro vectors,
which give rise to some extra symmetries.



\section{Virasoro vertex operator algebras and their extensions}

For complex numbers $c$ and $h$, we denote by $L(c,h)$ the
irreducible highest weight representation of the Virasoro algebra
with central charge $c$ and highest weight $h$.
It is shown in \cite{FZ} that $L(c,0)$ has a natural structure of a simple VOA.

\begin{df}
An element $e\in V$ is referred to as a {\it Virasoro vector of  central
charge\/} $c_e\in \C$ if $e\in V_2$ and it satisfies $e_{(1)}e=2e$ and
$e_{(3)}e=(c_e/2)\cd \vac$. It is well-known that the associated modes
$L^e(n):=e_{(n+1)}$,$n\in \Z$, generate a representation of the Virasoro
algebra on $V$ (cf.\ \cite{M1}), i.e., they satisfy the commutator
relation:
$$
  [L^e(m),L^e(n)]
  = (m-n)L^e(m+n)+\delta_{m+n,0}\dfr{m^3-m}{12}c_e.
$$
Therefore, a Virasoro vector together with the vacuum vector generates a
Virasoro VOA inside $V$. We will denote this subalgebra by $\vir(e)$.
A Virasoro vector $e$ is {\it simple} if  $\vir(e) \cong  L(c_e,0)$.
\end{df}

\subsection{Unitary Virasoro vertex operator algebras and their extensions}\label{sec:2.1}

Let
\begin{equation}\label{eq:2.1}
\begin{array}{rl}
  c_m &:= 1-\dfr{6}{(m+2)(m+3)},\qq m=1,2,\dots ,
  \vsb\\
  h_{r,s}^{(m)} &:= \dfr{\{ r(m+3)-s(m+2)\}^2-1}{4(m+2)(m+3)},\q
  1\leq s\leq r\leq m+1.
\end{array}
\end{equation}
It is shown in \cite{W} that $L(c_m,0)$ is rational and
$L(c_m,h_{r,s}^{(m)})$, $1\leq s\leq r\leq m+1$, provide all inequivalent
irreducible $L(c_m,0)$-modules (see also \cite{DMZ}).
This is the so-called unitary series of the Virasoro VOAs.
The fusion rules among $L(c_m,0)$-modules are computed in
\cite{W} and given by
\begin{equation}\label{eq:2.2}
  L(c_m,h^{(m)}_{r_1,s_1})\fusion_{L(c_m,0)} L(c_m,h^{(m)}_{r_2,s_2})
  = \dsum_{\scriptstyle i\in I \atop \scriptstyle j\in J}
    L(c_m,h^{(m)}_{\abs{r_1-r_2}+2i-1,\abs{s_1-s_2}+2j-1}),
\end{equation}
where
$$
\begin{array}{l}
  I=\{ 1\,,2,\,\dots,\,\min \{ r_1,\,r_2,\,m+2-r_1,\,m+2-r_2\}\} ,
  \vsb\\
  J=\{ 1,\,2,\,\dots,\,\min \{ s_1,\,s_2,\,m+3-s_1,\,m+3-s_2\}\} .
\end{array}
$$

\medskip

Among $L(c_m,0)$-modules, only $L(c_m,0)$ and
$L(c_m,h_{m+1,1}^{(m)})$ are simple currents. It is shown in
\cite{LLY} that $L(c_m,0)\oplus L(c_m,h^{(m)}_{m+1,1})$ forms a simple
current extension of $L(c_m,0)$.
Note that $h_{m+1,1}^{(m)}=m(m+1)/4$ is an integer if
$m\equiv 0,\,3 \pmod{4}$ and a half-integer if $m\equiv 1,\,2 \pmod{4}$.

\begin{thm}[\cite{LLY}]\label{thm:2.9}
  (1) The $\Z_2$-graded simple current extension
  $$
    \W(c_m)
    := L(c_m,0)\oplus L(c_m,h^{(m)}_{m+1,1})
  $$
  has a unique simple rational vertex operator algebra structure if
  $m\equiv 0,\,3 \pmod{4}$, and a unique simple rational vertex operator
  superalgebra structure if $m\equiv 1,\,2 \pmod{4}$, both of which
  extends $L(c_m,0)$.
\end{thm}

\begin{df}\label{df:2.10}
  Let $m\equiv 0$ or $3 \pmod{4}$.
  A simple $c=c_m$ Virasoro vector $u$ of a VOA $V$ is called
  {\it extendable\/} if there exists a non-zero highest weight vector
  $w\in V$ of weight $h^{(m)}_{m+1,1}=m(m+1)/4$ with respect to
  $\vir(u)$ such that the subalgebra generated by $u$ and $w$ is
  isomorphic to the extended Virasoro VOA $\W(c_m)$.
  We will call such a $w$ an {\it $h_{m+1,1}^{(m)}$-primary vector\/}
  associated to $u$.
\end{df}

\begin{lem}[cf. \cite{HLY2}]\label{lem:2.11}
  Let $m\equiv 0,\,3 \pmod{4}$ and $u\in V$ be a simple extendable
  $c=c_m$ Virasoro vector.
  Then a $h_{m+1,1}^{(m)}$-primary vector associated to $u$
  is unique up a scalar multiple.
\end{lem}

\begin{proof}
See Lemma 2.9 of \cite{HLY2}.
\end{proof}

In this article, we will mainly interested in the case $c_3=4/5$.
The following theorem is proved in \cite{KMY}.

\begin{thm}
The extended Virasoro VOA $\W(\sfr{4}{5})$ is rational and has
six inequivalent irreducible modules:
\[
  W[0],\  W[\sfr{2}5],\  W[\sfr{2}3]^+,\  W[\sfr{2}3]^-, \  W[\sfr{1}{15}]^+,\
  W[\sfr{1}{15}]^-,
\]
where $W[0]\cong  L(\sfr{4}{5},0)\oplus L(\sfr{4}{5},3)$,
$W[\sfr{2}5]\cong L(\sfr{4}{5},\sfr{2}{5})\oplus L(\sfr{4}{5},\sfr{7}{5})$,
$W[\sfr{2}3]^+\cong W[\sfr{2}3]^- \cong  L(\sfr{4}{5},\sfr{2}{3})$ and
$W[\sfr{1}{15}]^+\cong W[\sfr{1}{15}]^-\cong L(\sfr{4}{5},\sfr{1}{15})$ as
$L(\sfr{4}{5},0)$-modules.
The ambiguity on choosing signs $\pm$ is solved by fusion rules (cf.\ \cite{M2}).
\end{thm}

The fusion rules among irreducible $\W(\sfr{4}{5})$-modules
have some natural $\Z_3$-symmetries (cf.~\cite{M2,LLY}) and we can extend them
to automorphisms of VOAs containing these extended Virasoro VOAs.

\begin{thm}[\cite{M2,LLY}]\label{thm:2.12}
  Let $V$ be a VOA and let $U$ a sub VOA of $V$ isomorphic to $\W(\sfr{4}{5})$.
  Define a linear automorphism $\xi_U$ of $V$ to act on each
  irreducible $\W(\sfr{4}{5})$-submodule $M$ by
  $$
  \begin{cases}
    ~1 & \mbox{if}~~ M\cong W[0]
    ~~\mbox{or}~~ W[\sfr{2}5],
    \vsb\\
    ~e^{\pm 2\pi\sqrt{-1}/3} & \mbox{if}~~
    M\cong W[\sfr{2}{3}]^\pm ~~\mbox{or}~~
    W[\sfr{1}{15}]^\pm.
  \end{cases}
  $$
  Then $\xi_U$ defines an element in $\aut(V)$ and $\xi_U^3=1$.
\end{thm}

\begin{df}
  Let $u$ be a simple $c=4/5$ Virasoro vector in a VOA $V$.
  Let $V_u[h]$ be the sum of all irreducible $\vir(u)$-submodules of $V$ isomorphic
  to $L(\sfr{4}5,h)$.
  The Virasoro vector $u$ is said to be of $\sigma$-type on $V$ if $V_u[h]=0$
  for any $h\ne 0$, $3$, $2/5$, $7/5$.
\end{df}

\begin{prop}[\cite{M2}]\label{prop:2.8}
Let $u$ be a simple $c=4/5$ Virasoro vector of $\sigma$-type on $V$. Then
the  linear map $\sigma_u$ defined by
\begin{equation}\label{eq:2.4}
  \sigma_u:=
  \begin{cases}
    \, \ 1 & \mathrm{on}~~ V_u[0]\oplus V_u[\sfr{7}{5}],
    \vsb\\
    \, -1 & \mathrm{on}~~ V_u[3]\oplus V_u[\sfr{2}{5}].
  \end{cases}
\end{equation}
is an automorphism of $V$.
\end{prop}

By the classification of irreducible $\W(\sfr{4}{5})$-modules, we have the
following observation.

\begin{lem}\label{lem:2.9}
Let $u$ be a simple extendable $c=4/5$ Virasoro vector of $V$ and
$U$ the subalgebra isomorphic to $\W(\sfr{4}{5})$ generated by $u$ and
its $3$-primary vector.
Then $u$ is of $\sigma$-type if and only if the automorphism $\xi_U$
defined in Theorem \ref{thm:2.12} is trivial on $V$.
\end{lem}

We also need the following result:

\begin{lem}\label{lem:2.8}
  Let $V$ be a VOA with grading $V=\bigoplus_{n\geq 0}V_n$,
  $V_0=\C \vac$ and $V_1=0$, and let $u\in V$ be a Virasoro vector
  such that $\vir(u)\cong L(c_m,0)$.
  Then the zero-mode $o(u)=u_{(1)}$ acts on the Griess algebra
  of $V$ semisimply with possible eigenvalues $2$ and
  $h_{r,s}^{(m)}$, $1\leq s\leq r\leq m+1$.
  Moreover, if $h_{r,s}^{(m)}\ne 2$ for $1\leq s\leq r\leq m+1$
  then the eigenspace for the eigenvalue $2$ is one-dimensional,
  namely, it is spanned by the Virasoro vector $u$.
\end{lem}

\begin{proof}
See Lemma 2.6 of \cite{HLY1}.
\end{proof}


\section{$3$-transposition property of $\sigma$-involutions}\label{sec:3.2.3}

In this section, we will study $\sigma$-involutions associated to simple extendable
$c=4/5$ Virasoro vectors of $\sigma$-type in a VOA $V$.
We will show that they generate a $3$-transposition group in $\aut(V)$.
The main idea has be given in the arXiv preprint of Matsuo \cite{Ma2}.
Some technical details were also given but this part was removed in
the final published version.

Let $V$ be a VOA of OZ-type, i.e., $V_n=0$ for $n<0$, $V_0 =\C \vac$
and $V_1=0$.
Then the weight two subspace $V_2$ forms a commutative (non-associative) algebra,
called the \emph{Griess algebra} of $V$,
with respect to the product $ a\cdot b=a_{(1)}b $ and has an invariant form
defined by $(a|b)\vac =a_{(3)}b$ for $a, b\in V_2$.
We use $B$ to denote the Griess algebra of $V$.

Now let $e$ be a simple extendable $c=4/5$ Virasoro vector of $\sigma$-type and
$U$ the sub VOA generated by $e$ and its $3$-primary vector.
Then $\xi_U=1$ and
\[
  V= V_U[0]\oplus V_U[\sfr{2}5],
\]
where $V_U[h]$ denotes the sum of all irreducible $U$-submodules of $V$
isomorphic to $W[h]$, $h=0$ or $2/5$.
Then the Griess algebra $B$ can be decomposed as
\begin{equation}\label{eq:bitype}
  B= B_e(2)\oplus B_e(0)\oplus B_e(\sfr{2}5),
\end{equation}
where $B_e(h)=\{ a\in B \mid e_{(1)}a=ha\}$ and $B_e(2)=\C e$ (cf.~Lemma \ref{lem:2.8}).
Since the bilinear form is invariant, the decomposition \eqref{eq:bitype} is indeed
orthogonal decomposition.
\begin{rem}
If the Virasoro vector $e$ is not extendable, then the decomposition in \eqref{eq:bitype} is no longer valid since $e_1$ may also
have eigenvectors of eigenvalue $7/5$. In fact, such kind of examples exists  in the VOA $V_{\sqrt{2}A_2}^+$ (see for example
\cite{KMY}).
\end{rem}

\begin{nota}
For any $x\in B$, we use $x_e(h)$ to denote the component of $x$ in $B_e(h)$, i.e.,
\[
  x= x_e(0)+x_e(\sfr{2}5)+x_e(2),
\]
where $x_e(h)\in B_e(h)$.
\end{nota}

\begin{thm}\label{thm:3.4}
  Let $e$, $f$ be distinct simple extendable $c=4/5$ Virasoro vectors of $V$ of
  $\sigma$-type and consider the subalgebra $B(e,f)$ of the Griess algebra of $V$
  generated by $e$ and $f$.
  Then one of the following holds.
  \\
  (1)~ $(e|f)=0$ and $e\cd f=0$. In this case $B(e,f)=\C e \oplus \C f$.
  \\
  (2)~ $(e|f)=1/25$, $\sigma_e f= \sigma_f e$ and $e\cd f=\dfr{1}{5}(e+f-\sigma_e f)$.
  In this case $B(e,f)$ is 3-dimensional spanned by $e$, $f$, and $\sigma_e f$.
\end{thm}

\begin{proof}
Write $f=\lambda e+f_e(0)+f_e(\sfr{2}{5})$.
Since the decomposition \eqref{eq:bitype} is orthogonal,
$(e|f)=\lambda (e|e)=2\lambda /5$ and $\lambda=5(e|f)/2$.
Then
$$
  e\cd f=5(e|f) e+\dfr{2}{5}f_e(\sfr{2}{5}),~~~~~
  e\cd (e\cd f)= 10(e|f)e+\dfr{4}{25}f_e(\sfr{2}{5}).
$$
Therefore, we have
\begin{equation}\label{eq:3.2}
  e\cd (e\cd f)=8(e|f)e+\dfr{2}{5} e\cd f.
\end{equation}
By exchanging $e$ and $f$, we also have
\begin{equation}\label{eq:3.3}
  f\cd (f\cd e)=8(e|f)f+\dfr{2}{5} e\cd f.
\end{equation}
Then
\begin{equation}\label{eq:3.4}
  f\cd (f\cd e)=25(e|f)^2 e
  +  \underbrace{\dfr{2}{5} f_e(\sfr{2}{5})^2}_{\in \C e \oplus B_e(0)}
  + \underbrace{\dfr{2}{5} \l( f_e(0)\cd f_e(\sfr{2}{5})+6(e|f)f_e(\sfr{2}{5})\r)}_{\in B_e(2/5)} .
\end{equation}
Comparing the $B_e(\sfr{2}{5})$-part of the right hand side of \eqref{eq:3.3} and
that of \eqref{eq:3.4}, we get
\begin{equation}\label{eq:3.5}
  f_e(0)\cd f_e(\sfr{2}{5})=\l( \dfr{2}{5}+14(e|f)\r) f_e(\sfr{2}{5}).
\end{equation}
The $e_{(1)}$-decomposition of $f\cd f$ is:
$$
\begin{array}{ll}
  f\cd f
  &= \l( \dfr{5}{2}(e|f)e +f_e(0)+f_e(\sfr{2}{5})\r)
    \cd \l( \dfr{5}{2}(e|f)e +f_e(0)+f_e(\sfr{2}{5})\r)
  \vsb\\
  &= \dfr{25}{2}(e|f)^2 e +f_e(0)^2 +f_e(\sfr{2}{5})^2
    + \l( 30(e|f)+\dfr{4}{5}\r) f_e(\sfr{2}{5}).
\end{array}
$$
Therefore, by comparing $B_e(\sfr{2}{5})$-parts of $f\cd f=2f$ we obtain
\begin{equation}\label{eq:3.6}
  \l( 25(e|f)-1\r) f_e(\sfr{2}{5})=0.
\end{equation}
By the symmetry, we also have the same equality for $B_f(\sfr{2}{5})$.
\begin{equation}\label{eq:3.7}
  \l( 25(e|f)-1\r) e_f(\sfr{2}{5})=0.
\end{equation}
Suppose $25(e|f)-1\ne 0$.
Then $f_e(\sfr{2}{5})=f_e(\sfr{2}{5})=0$ and $5(e|f)f=e\cd f=f\cd e=5(e|f)e$.
Since $e\ne f$, we have $(e|f)=0$ and $e\cd f=0$.
This gives the statement (1) of the theorem.

So we may assume that $(e|f)=1/25$.
Since $\sigma_e f=f-2f_e(\sfr{2}{5})$, we have
$$
  e\cd f
  = \dfr{1}{5} e+\dfr{2}{5} f_e(\sfr{2}{5})
  = \dfr{1}{5} e+\dfr{2}{5} \cd \dfr{1}{2}(f-\sigma_e f)
  = \dfr{1}{5}(e+f-\sigma_e f).
$$
So we obtain
\begin{equation}\label{eq:3.8}
  \sigma_e f=e+f-5e\cd f=\sigma_f e.
\end{equation}
Therefore $B(e,f)$ is spanned by $e$, $f$ and $\sigma_e f$.
It remains to show that $e$, $f$ and $\sigma_e f$ are linearly independent.
For this, it suffices to show $f_e(\sfr{2}{5})\ne 0$.
Suppose contrary and $f_e(\sfr{2}{5})=0$.
Then $\sigma_e f=f$ and from \eqref{eq:3.8} we obtain $5f\cd e=e$.
Since $f$ has eigenvalues $0$, $2/5$ and $2$ in $B$, this is a contradiction.
Hence $B(e,f)$ is 3-dimensional when $(e|f)=1/25$.
This completes the proof.
\end{proof}

The next theorem is again proved by Matsuo \cite{Ma2}.
\begin{thm}
  Let $E$ be a set of simple extendable $c=4/5$ Virasoro vectors of
  $\sigma$-type in $V$.
  Let $D=\{\sigma_e\mid e\in E\}$ the set of $\sigma$-involutions associated
  to $E$ and $G$ the subgroup of $Aut(V)$ generated by $D$.
  Then $(G,D)$ is a $3$-transposition group, i.e., the order of
  $\sigma_e\sigma_f$ is bounded by $3$ for all $e,f\in E$.
\end{thm}

\begin{proof}
Let $e, f$ be distinct elements in $E$.
Then $(e|f)=0$ or $(e|f)=1/25$.
If $(e|f)=0$, then $\sigma_e f=f$ and
$\sigma_f=\sigma_{\sigma_e f}=\sigma_e\sigma_f\sigma_e$.
Therefore $\sigma_e$ commutes with $\sigma_f$.
If $(e|f)=1/25$, then $\sigma_e f=\sigma_f e$ and we have
\[
  (\sigma_e\sigma_f)^3
  = (\sigma_e\sigma_f\sigma_e)(\sigma_f\sigma_e\sigma_f)
  = \sigma_{\sigma_e f}\sigma_{\sigma_{f}e} =1.
\]
Therefore, $G$ is a 3-transposition group.
\end{proof}

\begin{thm}
Let $E$ be the set of simple extendable $c=4/5$ Virasoro vectors
of $\sigma$-type in $V$.
Then the map
\[
\begin{array}{cccc}
  \sigma: &E &\longto& \aut(V)
  \\
  & e& \longmapsto& \sigma_e
\end{array}
\]
is injective.
\end{thm}

\begin{proof}
Let $e\in E$ and $w^e$ its unique $3$-primary vector.
Then $\sigma_e w^e=-w^e$ and $\sigma_e$ is non-trivial.
Now let $f\in E$ be another element.
If $(e|f)=0$, then $e\cdot f=0$ and $e$ and $f$ are mutually commutative.
In this case the subalgebra generated by $e$ and $w^e$ is in the commutant of $\vir(f)$.
Thus, $\sigma_f w^e=w^e$ and $\sigma_e\neq \sigma_f$.
If $(e|f)\neq 0$ then we have $\sigma_f f=f\neq \sigma_e f$ and $\sigma_e \ne \sigma_f$.
Therefore, $\sigma$ is injective.
\end{proof}

\begin{thm}
Let $E$ be a set of simple extendable $c=4/5$ Virasoro vector of $\sigma$-type.
Let $G=\la \sigma_e\mid e\in E\ra$ be the group generated by the
corresponding $\sigma$-involutions.  Then $G$ is centerfree on the sub VOA
generated by $E$.
\end{thm}

\begin{proof}
Let $g\in Z(G)$. Then $\sigma_{ge}=g \sigma_e g^{-1} = \sigma_e$ for all $e\in
E$. Hence we have $ge=e$ by the previous theorem and $g$ acts as identity on
the sub VOA generated by $E$.
\end{proof}

\begin{rem}
An element $e\in E$ may no longer  be extendable on the sub VOA generated by $E$.
Nevertheless, the decomposition in \eqref{eq:bitype} is still valid.
\end{rem}



\section{Griess algebras generated by $2$ or $3$ Virasoro vectors}

Let $E$ be a set of simple extendable $c=4/5$ Virasoro vectors of $\sigma$-type and
let $G=\la \sigma_e\mid e\in E\ra$ be the subgroup generated by the corresponding
$\sigma$-involutions.
We will discuss the structures of the Griess algebras generated by $2$ or $3$
Virasoro vectors of $\sigma$-type.

\subsection{Griess algebra generated by two Virasoro vectors}
Let $e_1, e_2$ be two distinct simple extendable $c=4/5$ Virasoro vectors
of $\sigma$-type and $B$ the Griess algebra generated by $e_1$ and $e_2$.
Then by Theorem \ref{thm:3.4} $(e_1|e_2)=0$ or $1/25$.
\medskip

\paragraph{Case 1:}
$(e_1|e_2)=0$.
Then $e_1\cdot e_2=0$ and $G=\la \sigma_{e_1}, \sigma_{e_2}\ra \cong \Z_2\times \Z_2$.

In this case, $B=\Span\{ e_1, e_2\}$ has dimension $2$.
The conformal vector (i.e., 2 times the identity element in $B$) is
$\om = e_1+e_2$ and the central charge is $8/5$.

\medskip

\paragraph{Case 2:}
$(e_1|e_2)=1/25$.
Then $\sigma_{e_1}\sigma_{e_2}$ has order 3 and
$G=\la \sigma_{e_1}, \sigma_{e_2}\ra \cong \mathrm{S}_3$.

Let $e_3=\sigma_{e_1}e_2 = \sigma_{e_2}e_1$.
Then the Griess algebra $B=\Span\{ e_1, e_2, e_3\}$ and the multiplication is
given by
\[
  e_i\cdot e_j =
  \begin{cases}
    2e_i & \text{ if } i=j,
    \vsb\\
    \dfr{1}{5} (e_i+e_j-e_k) & \text{ if }  \{i,j,k\}=\{1,2,3\}.
  \end{cases}
\]
The conformal vector is $\om = 6(e_1+e_2+e_3)/5$ and the central charge is $2$.

Let $\eta=\om -e_1$.
Then $\eta$ is a Virasoro vector of central charge $6/5$.
Set $b=(e_2-e_3)$.
Then we have
\[
  {e_1}_{(1)}b
  = \frac{1}5\{ ( e_1+e_2-e_3)- (e_1+e_3-e_2)\}
  = \frac{2}5( e_2-e_3)
  =\frac{2}5 b,
\]
and
\[
  \eta_{(1)} b = (\om -e_1)_{(1)} b =\frac{8}5 b.
\]
Thus, $b$ is a highest weight vector of weight $(2/5, 8/5)$ of
the sub VOA $\vir(e_1)\otimes \vir(\eta)$.

\begin{rem}
  Note that the Griess algebra $B$ studied in Case 2 is isomorphic to the
  Griess algebra of $V_{\sqrt{2}A_2}^{\hat{\au}}$ (see Section \ref{sec:3}).
\end{rem}

\subsection{Griess algebras generated by three Virasoro vectors}

Let $V$ be a VOA of OZ-type and $e$, $f$, $g$ three distinct simple extendable
$c=4/5$ Virasoro vectors of $\sigma$-type.
Assume that $(e|f)=(f|g)=1/25$ and $g\notin B(e,f)$,
the Griess subalgebra generated by $e$ and $f$.
Then $G=\la \sigma_e, \sigma_f, \sigma_g\ra $ is a center-free $3$-transposition
group acting on the Griess algebra $B(e,f,g)$ generated by $e$, $f$ and $g$.
By \cite{CH}, $G\cong \mathrm{S}_4$ or $3^2{:}2$.

\medskip

\paragraph{Case 1:}
$G\cong \mathrm{S}_4$.
In this case, the $\sigma$-involutions correspond to the transpositions and
there are six transpositions.
Let $e_{ij}$ be the Virasoro vector associated to the transposition $(ij)$ for
$1 \leq i< j\leq 4$.
The Griess algebra $B=\Span\{ e_{ij}\mid  1 \leq i< j\leq 4\}$ and the multiplication
is given by
\[
  e_{ij}\cdot e_{k\ell} =
  \begin{cases}
    2e_{ij} & \text{ if }~ \{i,j\}=\{k,\ell\},
    \vsb\\
    \dfrac{1}{5} (e_{ij}+e_{k\ell} -e_{j\ell}) 
    & \text{ if }~   \abs{\{ i,j\} \cap \{ k,\ell\}}=1,
    \vsb\\
    0 &\text{ otherwise.}
  \end{cases}
\]
The conformal vector is 
$$
  \om = \dfr{5}{7}(e_{12}+e_{13}+e_{14}+e_{23}+e_{24}+e_{34})
$$
and the central charge is $24/7$.

\begin{rem}
In the next section, we will show that the above Griess algebra can be
realized in the VOA $V_{A_2\otimes A_3}^{\hat{\au}}$.
\end{rem}

\medskip

\paragraph{Case 2:}
$G\cong 3^2{:}2$.
In this case, $G$ has nine involutions.
Let $e_{00}$, $e_{01}$, $e_{02}$, $e_{10}$, $e_{11}$, $e_{12}$, $e_{20}$, $e_{21}$,
$e_{22}$ be the corresponding $c=4/5$ Virasoro vectors.
Then
\[
B=\Span\{e_{00}, e_{01}, e_{02}, e_{10}, e_{11}, e_{12}, e_{20}, e_{21}, e_{22}\}
\]
and the multiplication is given by
\[
  e_{ij}\cdot e_{k\ell}=
  \begin{cases}
    2e_{ij} & \text{ if } i=k,~ j=\ell,
    \\
    \dfrac{1}{5} ( e_{ij}+e_{k\ell}- e_{rs}) & \text{ otherwise, }
  \end{cases}
\]
where $i+k+r\equiv j+\ell+s\equiv 0 \mod 3$.  The conformal vector is
\[
\om= \frac{5}9( e_{00}+e_{01}+ e_{02}+ e_{10}+ e_{11}+ e_{12}+ e_{20}+ e_{21}+ e_{22})
 \]
 and the central charge is $4$.

\begin{rem}
In the next section, we will see that the above Griess algebra can be
realized in the VOA $V_{A_2\otimes A_2}^{\hat{\au}}$.
\end{rem}

\section{Some explicit examples}
In this section, we will study some explicit examples using lattice VOA.

\begin{nota}\label{hatL}
Let $L$ be a positive definite even lattice. 
We use $\hat{L}=\{\pm \ee^\al\mid \al\in L\}$ to denote the central extension 
of $L$ by $\pm 1$ such that $\ee^{\al} \ee^\be = (-1)^{(\al, \be)} \ee^{\be}\ee^\al$.
The twisted group algebra $\C \{ L\}$ is a central product of $\C$ and $\hat{L}$ 
over $\{ \pm 1\}$.
The lattice VOA has a structure $V_L=M(1)\tensor \C\{ L\}$ where
$M(1)$ is a free bosonic VOA associated to $\C L$ 
(cf.~\cite{FLM}).
For $\alpha \in L$ we denote $\alpha_{(-1)}\vac\in M(1)$ simply by $\alpha$.
\end{nota}

\subsection{Virasoro vectors in lattice type VOA $V_{\sqrt{2}A_2}^{\hat{\au}}$}\label{sec:3}

First we recall a construction of simple extendable Virasoro
vectors of $c=4/5$ from Dong et al.~\cite{DLMN}.

Let $\left\{ \alpha _{1},\alpha _{2}\right\} $ be a set of simple roots of a
lattice of type $A_{2}$ and put $\alpha_0=-(\alpha_1+\alpha_2)$.
Consider the lattice VOA $V_{\sqrt{2}A_2}$ associated to $\sqrt{2}A_2$.
Since $\sqrt{2}A_2$ is doubly even, we can use the usual group algebra 
$\C [\sqrt{2}A_2]$ instead of the twisted group algebra $\C\{ \sqrt{2}A_2\}$.
Namely, we put $V_{\sqrt{2}A_2}=M(1)\tensor \C [\sqrt{2}A_2]$ where $M(1)$ is 
the free bosonic VOA associated to $\C\tensor (\sqrt{2}A_2)$.

\begin{lem}[\cite{DLMN}]
  Let $\w$ be the conformal vector of $V_{\sqrt{2}A_2}$.
  The element
  \begin{equation}\label{conformal}
    u^0
    =\dfr{2}{5}\w 
    +\dfr{1}{5}\dsum_{i=0}^2 \l( \ee^{\sqrt{2}\alpha_i}+ \ee^{-\sqrt{2}\alpha_i}\r)
  \end{equation}
  is a simple $c=4/5$ Virasoro vector of $V_{\sqrt{2}A_2}$.
\end{lem}

Let $\au$ be the isometry of $A_2$ defined by
\begin{equation}\label{eq:5.2}
  \au: \al_1\longmapsto \al_2\longmapsto \al_0\longmapsto \al_1.
\end{equation}
Then $\au$ lifts to an automorphism $\hat{\au}$ of $V_{\sqrt{2}A_2}$ such that
\[
   \ee^{\pm \sqrt{2}\al_1}
   \longmapsto \ee^{\pm \sqrt{2}\al_2}
   \longmapsto \ee^{\pm \sqrt{2}\al_0}
   \longmapsto \ee^{\pm \sqrt{2}\al_1}.
\]
We also consider a lift $\theta \in \aut(V_{\sqrt{2}A_2})$ of the 
$(-1)$-map on $\sqrt{2}A_2$:
$$
  \theta : \ee^{\beta}\longmapsto \ee^{-\beta}~~\mbox{ for }~~\beta \in \sqrt{2}A_2.
$$
One can easily check that $\theta$ commutes with $\hat{\au}$.
The following result can be found in \cite{KMY}.

\begin{lem}
  The simple $c=4/5$ Virasoro vector $u^0$ defined above is extendable, 
  of $\sigma$-type, and fixed by $\theta$ and $\hat{\au}$.
\end{lem}

\begin{proof}
It is clear that $u^0$ is fixed by $\theta$ and $\hat{\au}$.
Let $\beta=\alpha_1-\alpha_2$ and define
\[
  q 
  =  \frac{1}{9} \beta_{(-1)} (\hat{\au} \beta)_{(-1)} (\hat{\au}^2\beta)_{(-1)}\vac
   -\dfr{1}{2} \dsum_{i=0}^2 \hat{\au}^i\l( \beta_{(-1)}\left( \ee^{\sqrt{2}\alpha_0}+ \ee^{-\sqrt{2}\alpha_0}\right)\r) .
\]
Then $q$ is a highest weight vector of $u^0$ of weight $3$.
Clearly $q$ is fixed by $\hat{\au}$, whereas $\theta$ negates it 
since $\theta$ commutes with $\hat{\au}$.
It is shown in Lemma 4.1 of \cite{KMY} that highest weights of irreducible 
$\vir(u^0)$-submodules of $V_{\sqrt{2}A_2}$ are 0, 2/5, 7/5 and 3.
Therefore $u^0$ is of $\sigma$-type.
\end{proof}

By this proposition, we can consider $\sigma_{u^0}$ in $\aut(V_{\sqrt{2}A_2})$.

\begin{lem}\label{lem:5.4}
  $\sigma_{u^0}=\theta$ in $\aut(V_{\sqrt{2}A_2})$.
\end{lem}

\begin{proof}
It follows from the decomposition in Lemma 4.1 of \cite{KMY} that $\sigma_{u^0}$ acts 
by $-1$ on the weight one subspace of $V_{\sqrt{2}A_2}$.
Then $\theta\sigma_{u^0}$ acts trivially on the weight one subspace and hence
$\theta\sigma_{u^0}$ is a linear automorphism.
Since $\theta$ and $\sigma_{u^0}$ commute, the product $\theta\sigma_{u^0}$ has order two 
and $\theta\sigma_{u^0}=(-1)^{\beta_{(0)}}$ for some $\beta\in (\sqrt{2}A_2)^*$. 
Since $\theta$ and $\sigma_{u^0}$ fixes $u^0$, so does 
$(-1)^{\beta_{(0)}}=\theta\sigma_{u^0}$.
On the other hand, one has 
$$
  (-1)^{\beta_{(0)}} u^0= \dfr{2}{5}\w 
  + \dfr{1}{5}\dsum_{i=0}^2 (-1)^{(\beta,\sqrt{2}\alpha_i)}\l( \ee^{\sqrt{2}\alpha_i} 
  +\ee^{-\sqrt{2}\alpha_i} \r) .
$$
Therefore it is necessary for $(-1)^{\beta_{(0)}}$ to fix $u^0$ that 
$(\beta,\sqrt{2}A_2)\in 2\Z$.
This shows $(-1)^{\beta_{(0)}}=\theta\sigma_{u^0}=1$ and we obtain $\sigma_{u^0}=\theta$.
\end{proof}

Set $e^\pm :=(1+\hat{\au}+\hat{\au}^2)\ee^{\pm \sqrt{2}\alpha_1}\in V_{\sqrt{2}A}$.
Then $e^\pm \in V_{\sqrt{2}A_2}^{\hat{\au}}$.
The Griess algebra of $V_{\sqrt{2}A_2}^{\hat{\au}}$ is 3-dimensional spanned
by the conformal element $\w$ and two 2-primary vectors $e^\pm$.
Its structure is described as follows.
$$
  e^\pm_{(1)}e^\pm = 2e^\mp,~~~~
  e^\pm_{(1)}e^\mp = 6\w,~~~~
  (\w|\w)= 1,~~~~
  (e^+|e^-) = 3,~~~~
  (e^\pm | e^\pm)=0.
$$
Set $\rho:=\zeta^{\sqrt{2}{\alpha_1}_{(0)}}$ with $\zeta=\exp(2\pii/3)$.
Then $\hat{\au}\rho\hat{\au}^{-1}=\zeta^{\sqrt{2}(\au \alpha_1)_{(0)}}
=\zeta^{\sqrt{2}{\alpha_2}_{(0)}}$ defines the same automorphism as $\rho$
on $V_{\sqrt{2}A_2}$ and $\rho$ preserves the fixed point subalgebra 
$V_{\sqrt{2}A_2}^{\hat{\au}}$.
We see $\theta^2=\rho^3=1$ and $\la \theta,\rho\ra \cong \mathrm{S}_3$ in $\aut(V_{\sqrt{2}A_2}^{\hat{\au}})$.

By an explicit computation, it is easy to determine all non-trivial Virasoro vectors
in $V_{\sqrt{2}A_2}^{\hat{\au}}$.

\begin{lem}\label{c45}
There are exactly 6 Virasoro vectors in $V_{\sqrt{2}A_2}^{\hat{\au}}$
other than the conformal vector.
$$
\begin{array}{ll}
  \mbox{Central charge $4/5$ : }
  &  u^0:=\dfr{2}{5}\w+\dfr{1}{5} (e^+ +e^-),~~~u^i=\rho^i u^0,~~~i=1,2;
  \vsb\\
  \mbox{Central charge $6/5$ : }
  & v^0:=\dfr{3}{5}\w-\dfr{1}{5} (e^+ +e^-),~~~v^i=\rho^i v^0,~~~i=1,2.
\end{array}
$$
\end{lem}

\begin{rem}
Note that the $c=4/5$ Virasoro vectors $u^0,u^1, u^2$ in
$V_{\sqrt{2}A_2}^{\hat{\au}}$ are all extendable and of $\sigma$-type.
\end{rem}

Set $w^0:=e^+-e^-$.
Then $w^0$ is a highest weight vector for $\vir(u^0)\tensor \vir(v^0)
\cong L(\sfr{4}{5},0)\tensor L(\sfr{6}{5},0)$ with highest weight
$(2/5,8/5)$.
One can verify that
$$
  w^0_{(1)}w^0=-6u^0-16v^0 \in \vir(u^0)\tensor \vir(v^0).
$$

\begin{lem}\label{a2group}
  $\aut(V_{\sqrt{2}A_2}^{\hat{\au}})=\la \theta,\rho\ra\cong \mathrm{S}_3$.
\end{lem}

\begin{proof}
Let $g\in \aut(V_{\sqrt{2}A_2}^{\hat{\au}})$.
Then $g$ acts on the 3-set $\{ u^0, u^1,u^2\}$ as a permutation.
Since the subgroup $\la \theta, \rho\ra$ acts on this 3-set as $\mathrm{S}_3$,
by replacing $g$ by $g\rho^i$ if necessary, we may assume that
$gu^i=u^i$ for $i=0,1,2$.
Then $g$ acts trivially on the Griess algebra of $V_{\sqrt{2}A_2}^{\hat{\au}}$.
Let $J$ and $J'$ be highest weight vectors for $\vir(u^0)\tensor \vir(v^0)$
with highest weights $(3,0)$ and $(0,3)$, respectively.
Since $J$ and $J'$ are unique highest weight vectors in the weight 3 subspace,
$g$ acts on $J$ and $J'$ by scalars.
It is shown in (3.17) of \cite{TY} that $V_{\sqrt{2}A_2}^{\la \sigma\ra}$ is 
generated by $u^0$, $v^0$, $w^0$, $J$ and $J'$.
Since $w^0$ is a common eigenvector for the zero-modes of $J$ and $J'$,
$g$ acts trivially on $J$ and $J'$ since $g$ fixes $w^0$.
Therefore $g=1$ and $\aut(V_{\sqrt{2}A_2}^{\sigma})=\la \theta,\rho\ra$.
\end{proof}

Since $\sigma_{u^0}=\theta$ inverts $\rho$, we have 
$\sigma_{u^i}=\sigma_{\rho^iu^0}=\rho^i\sigma_{u^0}\rho^{-i}
=\sigma_{u^0}\rho^{-2i}=\sigma_{u^0}\rho^i$.
In particular, $\rho =\sigma_{u^0}\sigma_{u^1}$.
So we see:

\begin{cor}
  $\aut(V_{\sqrt{2}A_2}^{\hat{\au}})=\la \sigma_{u^0},\sigma_{u^1}\ra$.
\end{cor}








\subsection{Lattice VOAs and their automorphism group}
Next we recall few facts about lattice VOAs and their automorphism groups.

\begin{df}
Let $L$ be an integral lattice.
The \emph{norm} (or {\it square norm}) of a vector $v$ is defined to be the value $(v,v)$.
For any $n\in \Z$, we define
\[
  L(n)=\{ \al\in L\mid (\al, \al) =n\}
\]
to be the set of all norm $n$ vectors in $L$.
\end{df}

\begin{df}\label{tX}
Let $X$ be a subset of a Euclidean space.
Define $t_X$ to be the orthogonal transformation which is $-1$ on $X$
and is $1$ on $X^\perp$.
\end{df}

\begin{df}[cf. \cite{GL}]\label{rssd}
A sublattice $M$ of an integral lattice $L$ is {\it RSSD (relatively semiselfdual)}
if and only if $2L\le M + \ann_L(M)$, where
$\ann_L(M):=\{ \alpha \in L \mid (\alpha, M)=0\}$.
This implies that $t_M$ maps $L$ to $L$ and is equivalent to this property
when $M$ is a direct summand.
The property that $2 M^* \le M$ is called {\it SSD (semiselfdual)}.
It implies the RSSD property, but the RSSD property is often more useful.
\end{df}

Next we recall the description of the automorphism groups of lattice VOA
from \cite{DN}.

\begin{nota}
Let $L$ be a positive definite even lattice.
Let $O(L)$ be the isometry group of $L$ and $O(\hat{L})$ the automorphism group of
$\hat{L}$.
Then by Proposition 5.4.1 of \cite{FLM} we have an exact sequence
\[
  1 \longto \hom (L,\Z/2\Z) \to  O(\hat{L})\to  O(L) \to  1.
\]
It is known that $O(\hat{L})$ is a subgroup of $\aut(V_L)$ (cf.~loc.~cit.).
Let 
$$
  N(V_L) = \l\la \exp(a_{(0)}) \mid a\in (V_L)_1 \r\ra
$$ 
be the subgroup generated by the linear automorphisms.
Since 
$$
  \sigma \exp(a_{(0)}) \sigma^{-1} = \exp((\sigma a)_{(0)})
$$ 
for any $\sigma\in \aut (V_L)$, $N(V_L)$ is a normal subgroup of $\aut(V_L)$.
\end{nota}

\begin{thm}[\cite{DN}]\label{aut}
Let $L$ be a positive definite even lattice.
Then
\[
  \aut (V_L) = N(V_L)\,O(\hat{L})
\]
Moreover, the intersection $N(V_L)\cap O(\hat{L})$ contains a subgroup
$\hom (L,\Z/2\Z)$ and the quotient $\aut (V_L)/N(V_L)$ is isomorphic
to a subgroup of $O(L)$.
\end{thm}

\begin{rem}\label{L2=0}
  If $L(2)=\varnothing$, then $(V_L)_1=\Span\{ \al_{(-1)}\vac \mid \al\in L\}$.
  In this case, the normal subgroup
  $N(V_L)= \{\exp(\lambda \al_{(0)}) \mid \al\in L,~\lambda \in \C\}$ is abelian and
  we have $N(V_L) \cap O(\hat{L})= \hom (L,\Z/2\Z)$ and $\aut (V_L)/N(V_L) \cong O(L)$.
  In particular, we have an exact sequence
  \begin{equation}\label{eq:5.3}
    1\longto N(V_L) \longto \aut(V_L) \stackrel{~\varphi~}\longto O(L)\longto 1.
  \end{equation}
  Note also that $\exp(\lambda \al_{(0)})$ acts trivially on $M(1)$ and
  $\exp(\lambda \al_{(0)})\ee^\be = \exp(\lambda (\al, \be)) \ee^{\be}$ for any
  $\lambda \in \C$ and $\al, \be \in L$.
\end{rem}

\begin{thm}\label{centralizer}
Let $L$ be an even positive definite lattice with $L(2)=\varnothing$.
Let $\au$ be a fixed point free isometry of $L$ of prime order $p$ and
$\hat{\au}$ a lift of $\au$ in $O(\hat{L})$.
Then we have an exact sequence
\[
  1\longto \hom(L/(1-\au)L, \Z_p)
  \longto \cent_{\aut(V_L)}(\hat{\au})
  \stackrel{\varphi}\longto \cent_{O(L)}(\au) \longto 1.
\]
\end{thm}

\begin{proof}
Consider the exact sequence in \eqref{eq:5.3}.
By Theorem \ref{aut} and Remark \ref{L2=0}, we have
$\aut(V_L)= N(V_L)\, O(\hat{L})$ and $\Ker\varphi =N(V_L)$.
By considering $O(\hat{L})$ as a subgroup of $\aut(V_L)$, we have
$\cent_{O(\hat{L})}(\hat{\au}) <  \cent_{\aut(V_L)}(\hat{\au})$ and clearly
$\varphi(\cent_{O(\hat{L})}(\hat{\au}))= \cent_{O(L)}(\au)$.
So it suffices to show that
$\Ker \varphi|_{\cent_{\aut(V_L)}(\hat{\au})} \cong \hom(L/(1-\au)L, \Z_p)$.

Suppose $\exp (h_{(0)}) \in \cent_{\aut(V_L)}(\hat{\au})$ for some $h\in \C L$.
By the conjugation relation $\hat{\au}\exp (h_{(0)})\hat{\au}^{-1}=\exp ((\hat{\au} h)_{(0)})$, 
the map $\exp(h_{(0)})$ centralizes $\hat{\au}$ if and only if
$(h, \be -\au\be )\in 2\pi\sqrt{-1}\,\Z$ for any $\be\in L$, which is
equivalent to that $h\in 2\pi\sqrt{-1} ((1-\au)L)^*$.
Since $\au$ is fixed point free of order $p$, we have $L> (1-\au)L > pL$
(see \cite{GL}).
Thus, $((1-\au)L)^* < \frac{1}p L^*$.
Therefore, $h=2\pii \gamma/p$ for some $\gamma \in L^*$ with
$(\gamma, (1-\au)L) \in p\Z$.
The map
\[
\begin{array}{lll}
  f_\gamma:& L  & \longto  \Z/p\Z
  \\
  & \al & \longmapsto  (\gamma,\alpha) \mod p
\end{array}
\]
defines a homomorphism in $\hom(L/(1-\au)L, \Z/p\Z)$.
Hence, we have obtained an isomorphism
$\exp(2\pii \gamma_{(0)}/p)\longmapsto f_\gamma$ between
$\Ker \varphi|_{\cent_{\aut(V_L)}(\hat{\au})}$ and $\hom(L/(1-\au)L, \Z_p)$.
\end{proof}

\begin{rem}\label{rhoa}
Let $\au$ be the fixed point free automorphism of $A_2$ in \eqref{eq:5.2}.
Then we have $(1-\au)A_2= 3A_2^*$.
Consider $L=\sqrt{2}A_2$ and the induced action of $\au$.
For any $\alpha \in \sqrt{2}A_2$, we have $(\al , (1-\au)\sqrt{2}A_2) \in 6\Z$.
Now let $\al$ be a norm 4 vector in $\sqrt{2}A_2$ and $\zeta =\exp(2\pii/3)$.
Define $\rho_\alpha:=\zeta^{\sqrt{2}\alpha_{(0)}}$.
Then $\rho_\al$ is an automorphism of $V_{\sqrt{2}A_2}^{\hat{\au}}$.
Set $e^\pm =(1+\hat{\au}+\hat{\au}^2)e^{\pm\alpha}$.
Then $\rho_{\al} (e^\pm) = \zeta^{\pm 1} e^\pm$.
Therefore, $\rho_\al$ agrees with the automorphism $\rho$ defined
in Section \ref{sec:3}.
\end{rem}

\begin{lem}\label{phise}
  Let $L$ be an even positive definite lattice with $L(2)=\varnothing$.
  Let $\au$ be a fixed point free isometry of $L$ of order $3$.
  Then for any $\al\in L(4)$, the sublattice
  $$
    A(\al)=\Span\{\al, \au(\al)\} \cong \sqrt{2}A_2.
  $$
  If $A(\al)$ is an RSSD sublattice of $L$, i.e., $2L \leq A(\al) + \ann_L(A(\al))$,
  then any $c=4/5$ Virasoro vector $u\in V_{A(\al)}^{\hat{\au}}$ is extendable and 
  of $\sigma$-type in $V_L^{\hat\au}$.
  Moreover, $\varphi(\sigma_u) = t_{A(\al)}$ where $\varphi$ is as in \eqref{eq:5.3}.
\end{lem}

\begin{proof}
Let $\alpha$ in $L(4)$.
Since $\au$ is fixed point free, $(1+\au +\au^2)\alpha =0$ and  
$((1+\au+\au^2)\alpha,\alpha)=0$.
As $\au$ is of order three, $\au^2\alpha=\au^{-1}\alpha$ and 
we obtain $(\alpha, \au \alpha)=-2$.
Therefore, $A(\alpha)$ is isomorphic to $\sqrt{2}A_2$.
Let $\pi: L \to A(\alpha)^*$ be the natural map.
Then by the assumption that $A(\alpha)$ is an RSSD in $L$,
we have $\pi(L)\subset \frac{1}{2} A(\alpha)$.
Let $u$ be a $c=4/5$ Virasoro vector of $V_{A(\alpha)}^{\hat{\au}}$.
Then $u$ is extentable by Lemma \ref{c45}.
Thus, by the decomposition of $V_{A(\al)}$-modules as $\vir(u)$-modules
in Lemma 4.1 of \cite{KMY} and Lemma 4.1 of \cite{LY}, 
we know that the highest weight vectors of $u$ in $V_{L}$ have 
weights $0$, $2/5$, $7/5$ or $3$.
Hence, $u$ is of $\sigma$-type in $V_L$.
Moreover, $\varphi(\sigma_u) $ acts as $-1$ on $A(\alpha)$ and $1$ on
$\ann_L(A(\al))$ by Lemma 5.4.  
Thus, $\varphi(\sigma_u) = t_{A(\al)}$ as desired.
\end{proof}


\subsection{Tensor products of $A_2$ and root lattices}

\begin{df}

Let $A$ and $B$ be integral lattices with the inner products $(~\,,~)_A$
and $(~\,,~)_B$, respectively.
\textit{The tensor product of the lattices} $A$ and $B$ is defined to be
the integral lattice which is isomorphic to $A\otimes_\Z B$ as a
$\Z$-module and has the inner product given by
$(\al\otimes \be, \al' \otimes\be') = (\al,\al')_A \cdot (\be,\be')_B$,
for any  $ \al,\al'\in A$,  and $\be,\be'\in B. $
We simply denote the tensor product of the lattices $A$ and $B$ by $A \otimes B$.
\end{df}

Let $R$ be a root lattice. 
Then the tensor product $A_2\otimes R$ is an even lattice, whose minimal norm is $4$.  
The following lemma is proved in Lemma 3.3 of \cite{GL}.

\begin{lem}\label{atb}
Let $R$ be a root lattice. Then the minimal vectors of $A_2\otimes R$ are given by
\[
  \{\al\otimes \be\mid  \al\in A_2,~ \be\in R,~(\alpha,\alpha)=(\beta,\beta)=2\}.
\]
\end{lem}

\begin{nota}
For each root $\be\in R$, we have 
$A_2\otimes \Z\be \cong A_2\otimes A_1 \cong \sqrt{2}A_2$. 
We will denote $A_2\otimes \Z\be$ by $\AR{\be}$. 
For simplicity, we also denote $A_2\otimes R$ by $\AR{R}$. 
Let $\au$ be the fixed point free isometry of $A_2$ in \eqref{eq:5.2}. 
By abuse of notation, we still denote $\au\tensor 1 \in O(\AR{R})$ by $\au$.
\end{nota}

\begin{lem}\label{AtensorR}
  Let $R$ be a root lattice. 
  Let $\au$ be the fixed point free isometry of $\AR{R}$ of order three induced by \eqref{eq:5.2}. 
  Then
  \\
  (1)~ $(1-\au)\AR{R} \cong 3A_2^*\otimes R$. 
  In particular, $[\AR{R}:(1-\au)\AR{R}]=3^{\rank R}$.
  \\
  (2)~ $\cent_{O(\AR{R})}(\au) \cong O(R)$.
\end{lem}

\begin{proof}
(1):~First we note that $(1-\au) A_2\cong 3A_2^*$. 
Since $(1-\au) (\al\otimes \be)= ( (1-\au) \al) \otimes \be$, we have 
$(1-\au)\AR{R} \cong 3A_2^*\otimes R$ and $[\AR{R}:(1-\au)\AR{R}]=3^{\rank R}$.
\\
(2):~Suppose $R\ne A_2$
Since $\AR{R}= A_2\otimes R$, we have $O(\AR{R})\cong O(A_2)*O(R)$, 
the central product of $O(A_2)$ and $O(R)$, by Lemma \ref{atb}. 
If $R=A_2$, then the flip of the arguments of the tensors is also  an isometry and we have $O(\AR{A_2})\cong (O(A_2)*O(A_2)){:}2$.
In either cases, we have 
$\cent_{O(\AR{R})}(\au) \cong O(R)$ since $\cent_{O(A_2)}(\au) = Z(O(A_2))=\la \pm 1\ra$.
\end{proof}

\begin{prop}\label{prop:5.20}
Let $R$ be a root lattice. 
Then $\cent_{\aut(V_{\AR{R}})}(\au) \cong  3^{\rank R}. O(R)$ and 
we have an exact sequence
\[
  1 \longto \hom(\AR{R}/(1-\au)\AR{R}, \Z/3\Z)
  \longto \cent_{\aut(V_{\AR{R}})}(\hat{\au})
  \overset{\varphi} {\longto} O(R) \longto 1.
\]
\end{prop}

\begin{proof}
It follows from Theorem \ref{centralizer} and Lemma \ref{AtensorR}.
\end{proof}

\begin{nota}\label{ER}
  Let $R$ be a root lattice. 
  For $\beta\in R(2)$, let $E(\beta)$ be the set of simple extendable $c=4/5$ 
  Virasoro vectors of $V_{\AR{\be}}^{\hat{\au}}$, and set 
  $E(R)= \cup_{\beta\in R(2)} E(\beta)$ in $V_{\AR{R}}$.
\end{nota}

\begin{prop}\label{weylpart}
  Let $R$ and $E(R)$ be as in Notation \ref{ER}. 
  Let $G= \la \sigma_u \mid u\in E(R)\ra$ and let 
  $\varphi: \cent_{\aut(V_{\AR{R}})}(\hat{\au}) \longto O(R)$ 
  be the natural map as in Proposition \ref{prop:5.20}. 
  Then $\varphi(G) \cong \mathrm{Weyl}(R)$, 
  where $\mathrm{Weyl}(R)$ is the Weyl group of $R$.
\end{prop}

\begin{proof}
Let $\be\in R(2)$ and take a simple $c=4/5$ Virasoro vector $u\in V_{\AR{\be}}^{\hat{\nu}}$. 
Since $\Z\be \cong A_1$, we have $(\Z\beta)^* =\frac{1}2 \Z\be$ and 
$\pi_\be (\AR{R}) \subset A_2\otimes \frac{1}2 \Z\be = \frac{1}2 \AR{\be}$, 
where $\pi_\be : \AR{R} \to \AR{\be}^*$ is the natural projection map. 
Hence, $\AR{\be}$ is an RSSD in $\AR{R}$. 
By Lemma \ref{phise}, $u$ is of $\sigma$-type in $V_{\AR{R}}^{\hat{\nu}}$ and
$\varphi(\sigma_u)= t_{\AR{\be}}$.
Let $r_\be$ be the reflection associated to the root $\be$ on $R$, i.e., 
\[
  r_\be \gamma = \gamma -(\beta, \gamma)\beta\quad \text{ for } \gamma\in R.
\]
We will prove that $t_{\AR{\beta}}=\id_{A_2}\tensor r_\beta$ on $\AR{R}=A_2\tensor R$.
Let $\alpha\in A_2$ be arbitrary and $\gamma\in R$ a root.
Then we have $(\beta,\gamma)=\pm 2$, $\pm 1$ or $0$.
Since $t_{\AR{\beta}}$ and $r_\beta$ are linear maps, replacing $\gamma$ by $-\gamma$ 
if necessary,  we may assume that $(\beta,\gamma)\geq 0$.

If $(\beta,\gamma)=2$, then $\gamma=\beta$ and 
$t_{\AR{\beta}} (\alpha\tensor \gamma) = (- \alpha)\tensor \gamma = \alpha\tensor (-\gamma) 
= \alpha \otimes r_\be \gamma$.
If $(\beta,\gamma)=0$, then $\alpha\tensor \gamma$ is in $\ann_{\AR{R}}(\AR{\beta})$ 
and $t_{\AR{\beta}} (\alpha\tensor \gamma) = \alpha\tensor \gamma = \alpha\tensor r_\beta\gamma$, also.

Now suppose $(\beta,\gamma)=1$. 
Then we have an orthogonal decomposition 
$$
  \alpha \tensor \gamma 
  = \dfr{1}{2}\alpha \tensor \beta
  + \dfr{1}{2}\alpha\tensor (2\gamma-\beta)
$$
with $\alpha\tensor \beta \in \AR{\beta}$ and  
$\alpha\tensor (2\gamma-\beta)\in \ann_{\AR{R}}(\AR{\beta})$.
From this we have 
$$
  t_{\AR{\beta}} (\alpha\tensor \gamma)
  = (-1)\cd \l(\dfr{1}{2}\alpha \tensor \beta\r) 
  + \dfr{1}{2}\alpha\tensor (2\gamma-\beta)
  = \alpha\tensor (\gamma -\beta )
  = \alpha \tensor r_\beta \gamma. 
$$
Since $\AR{\beta}$ is spanned by elements of the form $\alpha\otimes \gamma$,
$\alpha\in A_2$ and $\gamma\in R(2)$, 
$t_{\AR{\beta}}$ acts as $\mathrm{id}_{A_2}\otimes r_\be$ on $\AR{R}=A_2\otimes R$.
Therefore, $\varphi(G) = \la r_\be \mid \be \in R(2)\ra = \mathrm{Weyl}(R)$.
\end{proof}

\medskip

Next we will determine  the kernel $\Ker \varphi|_{G}$ of $\varphi$.

\begin{lem}\label{gen}
Let $R$, $E(R)$ and $G$ be defined as in Proposition \ref{weylpart}.
Then $\Ker \varphi|_{G}$ is generated by
\[
  \{ \rho_{\al\otimes \be}\mid  \al\in A_2,~ \be \in R,~ (\al,\al)=(\be,\be)=2 \},
\]
where $\rho_{\al\otimes \be}=\zeta^{(\al\tensor \be)_{(0)}}$ with $\zeta =\exp(2\pii/3)$.
\end{lem}

\begin{proof}
Fix a root $\al\in A_2$.  
Let $\be$ be a root of $R$ and take a simple extendable $c=4/5$ Virasoro vector 
$u_\be \in V_{\AR{\be}}^{\hat{\nu}}$. 
Then $u_\be $, $\rho_{\al\otimes \be}u_\be$ and $\rho_{\al\otimes \be}^2 u_\be$ 
are all simple extendable $c=4/5$ Virasoro vectors in $V_{\AR{\be}}^{\hat{\nu}}$ 
by Lemma \ref{c45} (see also Remark \ref{rhoa}).
By Lemma \ref{a2group} and Remark \ref{rhoa}, we also have
\[
  \sigma_{u_\be} \sigma_{\rho_{\al\otimes \be}u_\be} 
  = \sigma_{u_\be} \rho_{\al\otimes \be} \sigma_{u_\be} \rho_{\al\otimes \be}^{-1} 
  = \rho_{\al\otimes \be}^{-2}=\rho_{\al\otimes \be}.
\]
Since $E(R)= \cup_{\be \in R(2)} \{ u_\be , \rho_{\al\otimes \be}u_\be, \rho_{\al\otimes \be}^2 u_\be\}$, 
it is clear that $G$ is generated by 
$\{\rho_{\al\otimes \be}, \sigma_{u_\be}\mid \be \in R(2)\}$. 
Moreover, 
$\sigma_{u_\gamma} \rho_{\alpha\tensor \beta} \sigma_{u_\gamma}=\rho_{t_{\mathcal{A}_\gamma}(\alpha\tensor \beta)}$
for $\gamma \in R(2)$ by Lemma \ref{phise}.
Hence $\Ker \varphi|_{G}$ is generated by 
$\{ \rho_{\al\otimes \be} \mid \be \in R(2)\}$ as desired.
\end{proof}

\begin{lem}\label{Hom}
  Let $R$ be a root lattice and let 
  $P(\AR{R})= \{ x \in \AR{R}\mid (x, \AR{R})\in 3\Z\}$. 
  Then $\abs{\,\Ker \varphi|_{G}\,} =[\AR{R} : P(\AR{R})]$.
\end{lem}

\begin{proof}
Since $\rho: \AR{R}\longto \Ker \varphi_G$ is linear and surjective and the roots 
$\{\al\otimes \beta \mid \al\in A_2(2),~ \be \in R(2)\}$ spans $\AR{R}=A_2\otimes R$, 
it follows from Lemma \ref{gen} that $\rho_x\in \Ker \varphi_G$ if and only if 
$(x,\AR{R}) \in 3\Z$ for $x\in \AR{R}$.
Thus  $\big| \Ker \varphi|_{G}\big| =[\AR{R}: P(\AR{R})]$ as desired.
\end{proof}

\begin{lem}\label{KR}
  Let $R$ be a simple root lattice and let 
  $P(\AR{R})= \{ \gamma \in \AR{R}\mid (\gamma, \al)\in 3\Z\}$.
  \\
  (1) If $R\neq E_6$ or $A_n$ with $n+1\equiv 0\mod 3$, 
  then $[\AR{R}:P(\AR{R})]= 3^{\,\rank R}$.
  \\
  (2) If $R= E_6$ or  $R=A_n$ with $n+1\equiv 0\mod 3$, 
  then $[\AR{R}:P(\AR{R})]= 3^{\,\rank R-1}$.
\end{lem}

\begin{proof}
First we recall that $(1-\nu)(A_2\otimes R) = 3A_2^* \otimes R$ and hence 
$(1-\nu)\AR{R} < P(\AR{R})$.

Let $\al \in A_2$ be a root. 
Then $A_2= 3A_2^* \cup (\al +3A_2^*)\cup (-\al + 3A_2^*)$. 
Thus for any $\delta \in A_2\setminus 3A_2^* $ and $\be\in R$, we have
\[
  \delta\otimes \be \equiv \pm \al \otimes   \be \mod (1-\nu)\AR{R}.
\]
Therefore, every element in $\AR{R}\setminus (1-\nu)\AR{R}$ is congruent to 
$\al\otimes \gamma$ modulo $(1-\nu)\AR{R}$ for some $\gamma \in R$. 
Moreover, $\al\otimes \gamma \in P(\AR{R})$ if and only if 
\[
  (\al_i\otimes \be_j, \al\otimes \gamma) 
  = (\al_i, \al)\cdot (\be_j, \gamma ) \in 3\Z 
  \quad \text{ for all } i,j,
\]
where $\al_1, \al_2$ are simple roots of $A_2$ and $\be_1, \dots, \be_n$ are 
simple roots of $R$. 
Since $(\al_i, \al)=\pm 1 \mod 3$ for all $i$, we have $(\gamma, \be_j)\in 3\Z$ 
for all $j=1, \dots,n$ and hence $\gamma \in 3 R^*\cap R$. 
Therefore, we have 
$P(\AR{R}) = 3A_2^* \otimes R+A_2\otimes (3R^*\cap R)$.

If $R\neq E_6$ or $A_n$ with $n+1\equiv 0\mod 3$, then $[R^*:R]$ is relatively 
prime to $3$ and hence $3 R^*\cap R=3R$. 
Thus,  $P(\AR{R}) =3A_2^*\tensor R$ and 
$[\AR{R}:P(\AR{R})]=[A_2:3A_2^*]^{\,\rank R}=3^{\,\rank R}$.

If $R= E_6$ or $R=A_n$ with $n+1\equiv 0\mod 3$, then $3R^*\cap R \gneq 3R$ and hence 
$P(\AR{R}) = 3A_2^*\otimes R+A_2\otimes (3R^*\cap R) \gneq (1-\nu)\AR{R}$.
By the second isomorphism theorem, we have
\[
\begin{split}
  \frac{P(\AR{R})}{(1-\nu)\AR{R}} 
  &= \frac { 3A_2^* \otimes R + A_2\otimes (3R^*\cap R)}{3A_2^* \otimes R}
  \\
  & \cong \frac{A_2\otimes (3R^*\cap R)}{ (3A_2^* \otimes R) \cap (A_2\otimes (3R^*\cap R))}.
\end{split}
\]
Clearly $3A_2^* \otimes R \neq A_2\otimes (3R^*\cap R)$ and 
$(3A_2^* \otimes R) \cap (A_2\otimes (3R^*\cap R))$ contains $A_2 \otimes 3R$. 
Moreover, $[A_2\otimes (3R^*\cap R) : A_2\otimes 3R]= 3^2$ since $A_2$ has rank 2 
and $[3R^*\cap R:3R]=3$.

Let $a\in A_2^*\setminus A_2$ and $b\in R^*\setminus R$ such that $3b\in R\setminus 3R$. 
Then the element $3a\otimes 3b \in (3A_2^* \otimes R) \cap (A_2\otimes (3R^*\cap R))$ 
but it is not contained in $A_2\otimes 3R$. 
Therefore,
\[
  3^2 \gneq \abs{\frac{A_2\otimes (3R^*\cap R)}{ (3A_2^* \otimes R) 
  \cap (A_2\otimes (3R^*\cap R))}}\gneq 1
\]
and hence we must have 
$$
  \left| \frac{P(\AR{R})}{(1-\nu)\AR{R}}\right|
  =\left|\frac{A_2\otimes (3R^*\cap R)}{ (3A_2^* \otimes R)
  \cap (A_2\otimes (3R^*\cap R))}\right|=3.
$$ 
Therefore, $[\AR{R}:P(\AR{R})]= 3^{\,\rank R} / 3 = 3^{\,\rank R-1}$ as desired.
\end{proof}

\begin{thm}
Let $R$ be a simple root lattice and let $E(R)$ and $G$ be defined as in Proposition \ref{weylpart}.

(1) If $R\neq E_6$ or $A_n$ with $n+1\equiv 0\mod3$, then $G$ has the shape $3^{\rank R}. \mathrm{Weyl}(R)$

(2) If $R=E_6$ or $R=A_n$, $n+1\equiv 0\mod 3$, then $G$ has the shape $3^{\rank R-1}. \mathrm{Weyl}(R)$.
\end{thm}

\begin{proof}
 It follows from Proposition \ref{weylpart} and Lemmas \ref{Hom} and \ref{KR}.
\end{proof}

\begin{rem}
If we take $R=A_2$, then $V_{A_2\otimes A_2}^{\hat{\au}}$ has nine  extendable $c=4/5$ simple Virasoro vectors and the
subgroup generated by the corresponding $\sigma$-involutions has the shape $3.\mathrm{S}_3\cong 3^2:2$.
\end{rem}

\subsection{Lattices constructed from $\mathbb{F}_4$-codes}
Next  we discuss a construction of lattices from codes over $\F_4$
(cf.~\cite{cs, KLY}).

Let $\zeta :=(-1+\sqrt{-3})/2$ be a primitive cubic root of unity and let
$\mathcal{E}:= \Z[\zeta]$ be the ring of Eisenstein integers.
Then $\mathcal{E}/2\mathcal{E}=\{0, 1, \zeta, \zeta^2\} \cong \mathbb{F}_4$.
Let $\eta: \mathcal{E} \longto \mathcal{E}/2\mathcal{E}$ be the natural quotient map.

A linear $\F_4$-code $\mathcal{C}$ of length $n$ is a vector subspace of $\F_4^n$.
The weight $\wt (\alpha)$ of an element $\al=(\al_1, \dots, \al_n)\in \mathcal{C}$ is defined
to be the number of non-zero coordinates in $\al$.
We can define an Hermitian form on $\F_4^n$ by
\[
  \la (x_1, \dots, x_n),  (y_1, \dots, y_n)\ra =\sum_{i=1}^n x_i \ol{y_i}\in \F_4,
\]
where $\ol{0}=0$, $\ol{1}=1$, $\ol{\zeta}=\zeta^2$ and $\overline{\zeta^2}=\zeta$.

For any linear $\F_4$-code $\mathcal{C}$ of length $n$,  we define
\begin{equation}\label{eq:5.4}
  L_\mathcal{C} =\{ (x_1, \dots, x_n)\in \mathcal{E}^n \mid
  (\eta x_1, \dots, \eta x_n)\in \mathcal{C}\}.
\end{equation}
It is clear that $L_{\mathcal{C}}$ is a free $\Z$-module.
Moreover, one can define a real bilinear form on $L_{\mathcal{C}}$ by
$(u,v)= \mathrm{Re}\, \la u,v\ra$, where $\la \, , \,\ra $ is the canonical
Hermitian bilinear form on $\C^n$.
The \textit{norm} of a vector $v$ in $L_{\mathcal{C}}$ is defined by
$(v,v)=\la v, v\ra $.
Note also that the bilinear form $(~\,,~)$ is $\Z$-valued on $L_{\mathcal{C}}$.
In addition, $ 2\mathcal{E}\cong \sqrt{2}A_2$ as an integral lattice and hence
$L_{\mathcal{C}}$ as an integral lattice contains a sublattice isometric to $\sqrt{2}A_2^{n}$.
In the following, we will consider $L_{\mathcal{C}}$ as an integral lattice.

\begin{lem}
Let $\mathcal{C}$ be a linear $\F_4$-code and $L_\mathcal{C}$ the associated integral lattice.
\\
(1)~The lattice $L_{\mathcal{C}}$ is even if and only if  $\mathcal{C}$ is even,
i.e., $\wt(\al)$ is even for all $\al \in \mathcal{C}$.
\\
(2)~If the minimal norm of $\mathcal{C}$ is greater than or equal to $4$,
then $L_{\mathcal{C}}(2)=\varnothing$.
\end{lem}

\begin{df}
Let $\au$ be an isometry of the integral lattice $L_{\mathcal{C}}$ defined
by $\au v=\zeta v$ for any $v\in L_{\mathcal{C}}$.
Then $\au$ is fixed point free of order $3$.
\end{df}

\begin{nota}\label{Li}
Let $\hat{\au}$ be a lift of $\au$ in $\aut(V_{L_{\mathcal{C}}})$ and
denote $L=L_{\mathcal{C}}$.

For $i=0,1$, we denote
\[
  L^i=\{ \alpha \in L(4) \mid  \ee^\alpha_{(1)} \hat{\au}(\ee^{\alpha})
  = (-1)^i \hat{\au}^2(\ee^{-\al})\}.
\]
Note that
$\ee^\alpha \cdot \hat{\au}(\ee^{\alpha}) \in \{\pm \ee^{\al+\au\al}\}$
and $\ee^{\alpha +\au\alpha}=\ee^{-\au^2\al}$. 
\end{nota}

\begin{lem}\label{eA}
  Let $\al\in L^i$ and let $A(\al)=\Span_\Z\{\al , \au(\al)\}$.
  Let $\om_{A(\al)}$ be the conformal element of $V_{A(\al)}$.
  Then
  \[
    e_{A(\al)} = \frac{2}{5} \om_{A(\al)} +(-1)^i \frac{1}{5}
    \left(\ee^{\al}+ \ee^{-\al} + \hat{\au}(\ee^{\al} +\ee^{-\al})
    +\hat{\au}^2( \ee^{\al} + \ee^{-\al})\right)
  \]
  is a Virasoro vector of central charge $4/5$, which is fixed by $\hat{\au}$.
\end{lem}

\begin{proof}
By direct calculation, we have
\[
\begin{split}
  {e_{A(\al)}}_{(1)}e_{A(\alpha)} =
  & \frac{1}{25}\left[ 8\om_{A(\al)} +(-1)^i 8 \left(\ee^{\al}+ \ee^{-\al}
  + \hat{\au}(\ee^{\al} +\ee^{-\al})
  + \hat{\au}^2( \ee^{\al} + \ee^{-\al})\right) \right.
  \\
  & \left. \  +\left(\ee^{\al}+ \ee^{-\al} + \hat{\au}(\ee^{\al} +\ee^{-\al})
  + \hat{\au}^2( \ee^{\al} + \ee^{-\al})\right)^2 \right]
  \\
  =& \frac{1}{25}\left[ 8\om_{A(\al)} +(-1)^i 8 \left(\ee^{\al}+ \ee^{-\al}
  + \hat{\au}(\ee^{\al} +\ee^{-\al})+ \hat{\au}^2( \ee^{\al} + \ee^{-\al})\right) \right.
  \\
  & \  +2\left(\al(-1)^2+ (\au\al(-1))^2+ (\au^2 \al(-1))^2\right)
  \\
  & \ \left.  +(-1)^i 2 \left(\ee^{\al}+ \ee^{-\al} + \hat{\au}(\ee^{\al} +\ee^{-\al})
  + \hat{\au}^2( \ee^{\al} + \ee^{-\al})\right)  \right]
  \\
  = &\frac{1}{25} \left[ 20\omega_{A(\al)} +10 (-1)^i \left(\ee^{\al}+ \ee^{-\al}
  + \hat{\au}(\ee^{\al} +\ee^{-\al})+ \hat{\au}^2( \ee^{\al} + \ee^{-\al})\right) \right]
  \\
  =2e_{A(\al)}.
\end{split}
\]
Moreover,
\[
  (e_{A(\al)} | e_{A(\al)})
  = \frac{1}{25}\l( 4(\om_{A(\al)} | \om_{A(\al)}) + 6\r)
  = \frac{10}{25}
  = \frac{2}5
\]
as desired.
\end{proof}

\begin{lem}
Let $\mathcal{C}$ be an even $\F_4$-code with length $n$ and minimal norm $\geq 4$.
Then we have
$$
  \cent_{\aut(V_{L_\mathcal{C}})}(\au) \cong
  \hom(L_{\mathcal{C}}/(1-\au)L_{\mathcal{C}}, \Z_3).\cent_{O(L_\mathcal{C})}(\au)
  \cong 3^{n}. \cent_{O(L_\mathcal{C})}(\au) .
$$
\end{lem}

\begin{proof}
Since the minimal norm of $\mathcal{C}\geq 4$, we have
$L_{\mathcal{C}}(2)=\varnothing$.
Hence by Theorem\ref{centralizer}, we have an exact sequence
\[
  1 \longto \hom(L_{\mathcal{C}}/(1-\au)L_{\mathcal{C}},\Z_3)
  \longto \cent_{\aut(V_{L_\mathcal{C}})}({\au})
  \stackrel{\varphi}\longto \cent_{O(L_\mathcal{C})}(\au) \longto 1.
\]
It remains to show $[L_{\mathcal{C}}:(1-\au)L_{\mathcal{C}}]=3^n$.
Since $\au$ is fixed point free on $L_\mathcal{C}$,
$(1-\au)^2= 1-2\au+\au^2 =-3\au$ and we have the sequence
\[
  L_\mathcal{C}> (1-\au)L_\mathcal{C}> (1-\au)^2 L_\mathcal{C} =3L_\mathcal{C}
\]
and
\[
  [L_\mathcal{C}:(1-\au)L_\mathcal{C}] = [(1-\au)L_\mathcal{C}: 3L_\mathcal{C}].
\]
Since $L_{\mathcal{C}}$ has rank $2n$,
$3^{2n}=[L_\mathcal{C}: 3L_\mathcal{C}]=[ L_\mathcal{C}: (1-\au)L_\mathcal{C}]^2$.
Thus, $[L_\mathcal{C}: (1-\au)L_\mathcal{C}]=3^n$ as desired.
\end{proof}

\begin{lem}
Let $\mathcal{C}$ be an even $\F_4$-code.
For $\beta\in L_{\mathcal{C}}(4)$, set $A(\beta)=\Span\{\beta,\au\beta\}\simeq \sqrt{2}A_2$.
Let $E(\beta)$ be the set of extendable $c=4/5$ Virasoro vectors in 
$V_{A(\beta)}^{\hat{\au}}\subset V_{L_{\mathcal{C}}}^{\hat{\au}}$ 
and let $E$ be the union of $E(\beta)$, $\beta \in L_{\mathcal{C}}(4)$.
Let $G$ the subgroup generated by all $\sigma_u$, $u\in E$, and 
let 
$\varphi: \cent_{\aut(V_{L_\mathcal{C}})}(\au) \longto \cent_{O(L_\mathcal{C})}(\au)$
be the natural map.
Then we have $\varphi(G) =\la t_{A(\be)}\mid \be \in L_{\mathcal{C}}(4)\ra$.
\end{lem}

\begin{proof}
It follows from Lemma \ref{phise}.
\end{proof}

\subsection{The orbifold VOA $V_{K_{12}}^{\hat{\au}}$}

Next we will show that there is another kind of $c=4/5$ Virasoro vectors in the
lattice VOA $V_{K_{12}}$ associated to the 12-dimensional Coxeter-Todd lattice
$K_{12}$.

First let us recall a construction of the Coxeter-Todd lattice using the Hexacode
over $\mathbb{F}_4$.
The Hexacode $\mathcal{H}$ is a length $6$ self-dual $\F_4$-code generated by
\[
\begin{bmatrix}
  0&0&1&1&1&1
  \\
  1&0&1&0&\zeta& \zeta^2
  \\
  1&0&0&1&\zeta^2& \zeta
\end{bmatrix}_.
\]
The minimal weight of $\mathcal{H}$ is $4$ and the Hamming weight
enumerator is
\[
x^6+45x^2y^4 + 18y^6.
\]

The Coxeter-Todd lattice can be defined as the sublattice
\begin{equation}\label{dK12}
  K_{12}= L_{\mathcal{H}}
  = \{ (x_1, \dots, x_6)\in \mathcal{E}^6 \mid
  (\eta x_1, \dots, \eta x_6)\in \mathcal{H}\}.
\end{equation}

The following facts about $K_{12}$ can be found in \cite{CS2}.

\begin{lem}\label{PK12}
  Let $K_{12}$ be the integral lattice $L_{\mathcal{H}}$ associated with the Hexacode $\mathcal{H}$.
  \\
  (1)~ The isometry group $O(K_{12})$ of $K_{12}$ has the shape
  $(6\times \mathrm{PSU}(4,3).2).2$.
  \\
  (2)~ $K_{12}$ has $756$ vectors of norm $4$, $4032$ vectors of norm $6$ and
  $20412$ vectors of norm $8$.
  Moreover, $O(K_{12})$ acts transitively on vectors of norm $4$, $6$ and $8$,
  respectively.
  \\
  (3)~ Let $\au$ be an isometry of $K_{12}$ defined by the complex multiplication 
  $\au v=\zeta v$ for any $v\in K_{12}$.
  Then $\au$ has order $3$ and is fixed point free on $K_{12}$.
  Moreover, $O_3(O(K_{12}))=\la \au \ra$.
\end{lem}




Next we study the extendable $c=4/5$ Virasoro vectors in $V_{K_{12}}$.
In addition to the Virasoro vector $\rho_x e_{A(\al)}$ defined in Lemma \ref{eA},
there exists another kind of extendable $c=4/5$ Virasoro vectors in $V_{K_{12}}$.

\subsubsection{Dihedral $\sqrt{2}E_8$ pairs and Coxeter Todd lattice}\label{app:A}

Next we will recall some results about dihedral $EE_8$ pairs from
\cite{GL}. 


\begin{nota}\label{Q}
Let $M$ and $N$ be $\sqrt{2}E_8$-sublattices of the Leech lattice $\Lambda$ such that
$Q:=M+N$ is isometric to $\dih{6}{14}$ as obtained in \cite{GL}.
Let $t_M$ and $t_N$ be the SSD involutions associated to $M$ and $N$,
respectively (see Definition \ref{tX}).
Define $F:=M\cap N$, $J:=\ann_{Q}(F)$, $S^1:= M\cap J$ and $S^2:= N\cap J$.
The following facts can be found in Section 6.1 of \cite{GL}.
\\
(1)~$F=M\cap N \cong \sqrt{2}A_2$;
\\
(2)~$S^1\cong S^2 \cong \sqrt{2}E_6$ and $S^1\cap S^2=0$;
\\
(3)~$J=\ann_{Q}(F)\cong K_{12}$, the Coxeter-Todd lattice of rank 12;
\\
(4) Set $\au:= t_Mt_N$. 
Then $\au$ has order $3$. 
Moreover, $\au$ is fixed point free on $J$ and is contained in $O_3(O(J))$.
\end{nota}

For explicit calculation, we will fix a 2-cocycle on $Q$.

\begin{nota}\label{cocycle}
  Let $\al, \be \in Q$. Then $\al=\alpha_1+\alpha_2$, $\be =\beta_1+\beta_2$ for some
  $\alpha_1$, $\alpha_2\in M$ and $\beta_1$, $\beta_2\in N$.
  We define
  \[
    \varepsilon(\al, \be) = (-1)^{(\alpha_2, \beta_1)}.
  \]
  It is straightforward to show that $\varepsilon(\ , \ )$ is a bilinear 2-cocycle
  satisfying
  \[
    \varepsilon(M,M)=\varepsilon(N,N)=1,~~~
    \varepsilon(\alpha,\beta)\varepsilon(\beta,\alpha)=(-1)^{(\alpha,\beta)}
    \quad  \text { for any } \al, \be \in Q.
  \]
\end{nota}

\begin{nota}\label{eM}
Let $L$ be an $\sqrt{2}E_8$-sublattice of $Q$ such that $\varepsilon|_{L\times L}$ is
trivial (e.g., $L=M$ or $N$).
Let
\[
  e_L=\frac{1}{16} \om_L +\frac{1}{32} \sum_{\al\in L(4)} \ee^\al,
\]
where $\om_L$ is the conformal element of the lattice sub VOA $V_L$. Then
$e_L$ is an Ising vector.
For $x\in L^*$, define a $\Z$-linear map $\varphi_x=(-1)^{x_{(0)}} \in \aut(V_L)$:
\begin{equation*}\label{phisubx}
\varphi_x(u\otimes \ee^\al) =(-1)^{(x,\al)} u\otimes \ee^\al \quad
\text{ for } u\in M(1)\text{ and } \al \in L.
\end{equation*}
Then $ \varphi_x e_L$ is also an Ising vector of $V_L$.
\end{nota}

Let $e_M\in V_M$ and $e_N\in V_N$ be the Ising vectors defined as in Notation \ref{eM}.
Then we have $(e_M|e_N)=13\cd 2^{-10}$ (see \cite{GL}).

\medskip

The following lemma can be found in \cite{SY} (see also \cite{GL,LYY2}).

\begin{lem}\label{uv}
  Let $U$ be the sub VOA generated by $e_M$ and $e_N$. 
  Then $U$ contains an extendable $c=4/5$ Virasoro vector $u$ and an extendable 
  $c=6/7$ Virasoro vector $v$ such that $u$ and $v$ are both fixed by 
  $\hat{\au} = \tau_{e_M}\tau_{e_N}$ and
  \[
    U\cong \left( L(\sfr{4}{5},0)\oplus L(\sfr{4}{5},3)\right) 
    \otimes \left( L(\sfr{6}{7},0)\oplus L(\sfr{6}{7},5)\right) 
    \oplus \l(L(\sfr{4}{5},\sfr{2}{3})\otimes L(\sfr{6}{7},\sfr{4}{3})\r)^{\oplus 2}
  \]
  as an $L(\sfr{4}{5},0)\otimes L(\sfr{6}{7},0)$-module. 
  Moreover, $\la \hat{\au} \ra= \la \xi_W\ra$, where $W\cong \W(4/5)$ is the 
  $W_3$-algebra generated by $u$ and its 3-primary vector.
\end{lem}

\begin{rem}\label{defu}
The extendable $c=4/5$ Virasoro vector $u$ in Lemma \ref{uv} can be defined 
explicitly as follows (see \cite{SY}):
\[
  u = \frac{2^6}{135} \l( 2e_M +2e_N + \hat{\au}e_N -16(e_M)_{(1)}e_N\r) .
\]
\end{rem}

\begin{lem}\label{inK12}
The extendable $c=4/5$ Virasoro vector $u$ is contained in $V_J\cong V_{K_{12}}$.
\end{lem}

\begin{proof}
Recall that $F=M\cap N\cong \sqrt{2}A_2$.
Since $J=\ann_{Q}(F)$, it suffices to prove that $(\om_F|u)=0$.
Set $S^3=\au S^2$. 
We also denote $T^1= M\setminus (F\perp S^1)$, $T^2= M\setminus (F\perp S^2)$ 
and $T^3=\au T^2$. 
Then $M(4)=F(4)\cup S^1(4)\cup T^1(4)$ and $N(4)=F(4)\cup S^2(4)\cup T^2(4)$.
Since $F\perp S^1$ and $F\perp S^2$ are full sublattices of $M$ and $N$, 
we have $\w_M=\w_F +\w_{S^1}$, $\w_N=\w_F +\w_{S^2}$ and 
\[
\begin{split}
  e_M = \frac{1}{16}(\om_F+ \om_{S^1}) +\frac{1}{32}\left(
    \sum_{\al\in F(4)} \ee^\al + \sum_{\al\in S^1(4)} \ee^\al +\sum_{\al\in T^1(4)}
    \ee^\al \right),
  \\
  e_N = \frac{1}{16}(\om_F+ \om_{S^2}) +\frac{1}{32}\left(
    \sum_{\al\in F(4)} \ee^\al + \sum_{\al\in S^2(4)} \ee^\al
    +\sum_{\al\in T^2(4)} \ee^\al\right).
\end{split}
\]
Since $F\perp S^1$ is index three in $M$, there exist 
$\lambda\in \frac{1}{3}F$ and $\mu\in \frac{1}{3} S^1$
such that 
$$
  M=(F\perp S^1)\cup (\lambda+\mu+F\perp S^1)\cup(-\lambda -\mu +F\perp S^1).
$$
Moreover, one has $(\lambda,\lambda)=4/3$ and $(\mu,\mu)=8/3$ and
$$
  (\w_F)_{(1)} \ee^\alpha = \dfr{(\lambda,\lambda)}{2} \ee^\alpha = \dfr{2}{3} \ee^\alpha
  \mbox{~~for~~}  \alpha \in T^1(4).
$$
Then 
\[
\begin{array}{ll}
  ({\om_F}_{(1)}) e_M
  &= \dfrac{1}{16}\cd 2\om_F +\dfrac{1}{32}\left(2\dsum_{\al\in F(4)} \ee^\al 
     + \dfr{2}{3}\dsum_{\al\in T^1(4)} \ee^\al \right) 
  \vsb\\
  &= \dfrac{1}{8} \w_F+\dfr{1}{16}\dsum_{\al\in F(4)} \ee^\al 
    + \dfr{1}{48}\dsum_{\al\in T^1(4)} \ee^\al.
\end{array}
\]
Therefore,
\[
  (\om_F| (e_M)_{(1)}(e_N)) = \l((\om_F)_{(1)} e_M| e_N\r) 
  = \frac{1}{2^7} \cd 1 + \frac{1}{2^{10}} \cd 6 = \frac{5}{2^8}
\]
and hence
\[
  (\om _F| u) = \frac{2^6}{135}(\om_F |  2e_M +2e_N + \hat{\au}(e_N) -16(e_M)_{(1)}(e_N))
  = \frac{2^{6}}{135}\left( 5\cd\frac{1}{16} - 16\cd \frac{5}{2^8}\right) =0
\]
as desired.
\end{proof}

\begin{rem}
  Since $\la \hat{\au}\ra=\la \xi_W\ra$, it is clear that $u$ is of $\sigma$-type 
  in $V_{K_{12}}^{\hat{\au}}$. 
  In Appendix A, we will also give an explicit form of $u$ in $V_{K_{12}}^{\hat{\au}}$.
\end{rem}





\begin{lem}\label{3to6}
  We have $[K_{12}: (1-\au)K_{12}]=3^6$.
\end{lem}

\begin{proof}
Since $\au$ is fixed point free on $K_{12}$, we have $ (1-\au)^2=1-2\au+\au^2=-3\au$ 
as a linear map on $K_{12}$. 
Thus we have $(1-\au)^2K_{12}= 3K_{12}$. 
Therefore, we have the sequence
\[
  K_{12}> (1-\au)K_{12}> (1-\au)^2 K_{12} =3K_{12}
\]
and
\[
  [ K_{12}: (1-\au)K_{12}] = [(1-\au)K_{12}: 3K_{12}].
\]
Since $[K_{12}: 3K_{12}]= 3^{12}$, we have $[K_{12}: (1-\au)K_{12}]=3^6$ as desired.
\end{proof}

Let $f\in \hom(K_{12}/(1-\au)K_{12}, \Z_3)$. 
Then by Theorem \ref{centralizer}, the linear map $\rho_f$ defined by
\[
  \rho_f (u\otimes \ee^\al)
  = \zeta^{f(\al)} u\otimes \ee^\al \quad \text{ for } u\in M(1),~ \al\in K_{12},~
  \zeta =\exp(2\pii/3),
\]
induces an automorphism of $V_{K_{12}}^{\hat{\au}}$. 
Therefore, $\rho_f u$ is also an extendable $c=4/5$ Virasoro vector of $\sigma$-type 
in $V_{K_{12}}^{\hat{\au}}$.

\begin{lem}
  There are at least $1107$ extendable $c=4/5$ Virasoro vectors of $\sigma$-type in 
  $V_{K_{12}}^{\hat{\au}}$. 
  There are $3^6$ Virasoro vectors of the form $\rho_f u$, 
  $f\in \hom(K_{12}/(1-\au)K_{12},\Z_3)$ and $3\times 126$  Virasoro vectors 
  associated to the $126$ $\au$-invariant $\sqrt{2}A_2$-sublattice of $K_{12}$.
\end{lem}

\begin{nota}
  Let $\mathcal{A}$ be the set of all $\nu$-invariant $\sqrt{2}A_2$ sublattices in $K_{12}$. 
  Denote $E_1=\{ \rho_\al e_A\mid A\in \mathcal{A}, ~\al \in A(4)\}$ and 
  $E_2=\{ \rho_f u \mid f\in   \hom(K_{12}/(1-\au)K_{12}, \Z_3)\}$.
\end{nota}

\begin{rem}
The coset representatives of $(1-\au)K_{12}$ in $K_{12}$ have been computed in \cite[Page 427]{CS2}. 
Every coset can be represented by a vector of norm $0$, $4$, $6$, or $8$ and
\begin{center}
  the vectors of norm $0$ fall into $1$ class of size $1$; 
  \\
  the vectors of norm $4$ fall into $252$ classes of size $3$; 
  \\
  the vectors of norm $6$ fall into $224$ classes of size $18$; 
  \\
  the vectors of norm $8$ fall into $252$ classes of size $81$.
\end{center}
\end{rem}

\medskip

The next lemma can be obtained easily by using Equation \eqref{dK12} and (2) of 
Lemma \ref{PK12}.
\begin{lem}\label{orthal}
Let $\be\in K_{12}$ be of norm $4$, $6$, or $8$ and let
\[
  n_\be =| \{ \al\in K_{12}(4) | (\al, \be )=0\mod 3\}|.
\]
Then,
\[
n_\be =\begin{cases}
270  & \text{ if } (\be , \be)=4,\\
270 & \text{  if } (\be , \be)=6, \\
216 & \text{  if } (\be , \be)=8.
\end{cases}
\]
\end{lem}

\begin{cor}\label{orthA}
Let $\al\in K_{12}(4)$. Then $A(\al)=\Span\{ \al, \au(\al)\} \cong \sqrt{2}A_2$  and there exists exactly $80$ $\au$-invariant
$\sqrt{2}A_2$-sublattices $B\in \mathcal{A}$ such that $B\neq A(\al)$ and $B\not\perp A(\al)$.
\end{cor}

\begin{proof}
The first assertion is clear by definition. By Lemma \ref{orthal}, there 
exists 270 norm 4 vectors which are orthogonal to $A(\al)$. 
Note also that $(\al, \be )\equiv (\al, \au\be )\mod 3$ since 
$\frac{1}{3}(1-\au)K_{12} < K_{12}$ (see the proof of Lemma\ref{3to6}).
Thus there exists exactly $480$ norm 4 vectors which are not orthogonal 
to $A(\al)$ and not in $A(\al)$. 
Therefore we have$480/6=80$ $\au$-invariant $\sqrt{2}A_2$-sublattices 
$B\in \mathcal{A}$ such that $B\neq A(\al)$ and $B\not\perp A(\al)$.
\end{proof}

\begin{nota}
Let $(V, Q )$ be a nondegenerate orthogonal space over $\F_3$ and $\dim V=n$.

Up to equivalence, there are exactly two choices for the form $Q$ on $V$. 
These two types are distinguished by their discriminant (the determinant 
of their Gram matrix) $\delta = \pm 1$, or, in even dimension $n=2k$, by 
their sign $\epsilon = \pm  1$, where $\epsilon = + 1$ if $Q$ has Witt 
index $n$ and $\epsilon = - 1$ if $Q$ has Witt index $n - 1$. 
Note that $\delta\epsilon=(-1)^{n/2}$ if $\dim V=n$ is even.  
In odd dimension, both forms have the same Witt index, and the sign is not 
often defined. 
As a convention,  we choose  $\epsilon$ so that 
$\delta\epsilon = (-1)^{\frac{n+1} 2}$ if $\dim V=n$ is odd.

We denote the full orthogonal group of $(V, Q)$ by 
${}_\delta \mathrm{O}^\epsilon(n,3)$ if $(V,Q)$ has the discriminant 
$\delta$ and the sign $\epsilon$. 
Since $\delta$  and $\epsilon$ determine each other uniquely for any given $n$, 
we sometimes write $\mathrm{O}^\epsilon(n,3)$ instead of 
${}_\delta \mathrm{O}^\epsilon(n, 3)$. 
We also denote the derived subgroup 
$[{}_\delta \mathrm{O}^\epsilon(n,3), {}_\delta \mathrm{O}^\epsilon(n,3)]$ 
by $\Omega^\epsilon(n,3)$. 
Moreover, we often use $\pm$ to denote $\pm 1$.

The orthogonal group $\mathrm{O}^\epsilon(n,3)$ contains two conjugacy classes of
reflections, one containing those reflections $r_v$ whose reflection center
$\la v\ra$ has $Q(v ) = 1$ and the other containing those reflections 
$r_v$ whose center has $Q(v)=-1$. 
The reflections of the first class is said to be of $+$-type and those of 
the second class is of $-$ -type.
We denote the subgroup of $\mathrm{O}^\epsilon(n,3)$ generated by all reflections 
of $\gamma$-type by ${}^\gamma\Omega^\epsilon(n,3)$ and denote 
the central quotient 
${}^\gamma\Omega^\epsilon(n,3)/ ({}^\gamma\Omega^\epsilon(n,3)\cap \la \pm 1\ra)$ 
by ${}^\gamma \mathrm{P}\Omega^\epsilon (n,3)$.
\end{nota}

\begin{lem}
  Let $H=\la \sigma_e\mid e\in E_1\ra < \aut(V_{K_{12}})$. 
  Then $H\cong 3^6{:} (6\times \mathrm{PSU}(4,3).2)$.
\end{lem}

\begin{proof}
First we note that $\rho_\al e_A$ is of $\sigma$-type in $V_{K_{12}}$ and fixed by $\hat{\au}$. 
Thus, $\sigma_e\in C_{\aut(V_{K_{12}})}(\hat{\au})$ for any $e\in E_1$. 
By Lemma \ref{centralizer}, we have an exact sequence
\[
  1 \longto \hom(K_{12}/(1-\au)K_{12}, \Z_3) 
  \longto  C_{\aut(V_{K_{12}})}(\hat{\au}) 
  \xrightarrow{~\varphi~} C_{O(K_{12})}(\au) 
  \longto 1.
\]
Note also $C_{O(K_{12})}(\au)$ has the shape $6\times \mathrm{PSU}(4,3).2$ (see \cite{CS2}).
By Lemma \ref{phise}, we have $\varphi(\sigma_{e})= t_A$ for any $e\in V_A^{\hat{\au}}, A\in \mathcal{A}$. Thus,
$$
  \varphi(H) =\la t_A\mid  A\in \mathcal{A}\ra \cong 6\times \mathrm{PSU}(4,3).2.
$$
Therefore, the map $\varphi: H \to C_{O(K_{12})}(\au)$ is a surjection.
Moreover,
$\sigma_{\rho_\al e_A} \sigma_{e_A} = \rho_\al^2$ and hence the kernel of 
$\varphi$ contains the subgroup
$$
  \la \rho_\al \mid \al\in K_{12}(4)\ra \cong   \hom(K_{12}/(1-\au)K_{12}, \Z_3) \cong 3^6.
$$
Thus, $H$ has the desired shape.
\end{proof}

\begin{rem}
We should note that $H$ also acts on $V_{K_{12}}^{\hat{\au}}$  with the kernel $\la \hat{\au}\ra $. Hence, as a subgroup of
$\aut(V_{K_{12}}^{\hat{\au}})$, we have
\[
  \bar{H}=  \la\sigma_e|_{_{\aut(V_{K_{12}})}} \mid e\in E_1\ra \cong H/\la
  \hat{\au}\ra \cong  3^6{:}(2\times \mathrm{PSU}(4,3).2) \cong 3^6{:} {^+\Omega^-}(6,3).
\]
\end{rem}

\begin{lem}
Let $\be \in K_{12}$ and $A\in \mathcal{A}$. Then
\[
  (\rho_\be u|e_A) =
  \begin{cases}
    ~0 & \text{ if } (\be , A) \equiv 0 \mod 3, 
    \vsb\\
    \dfrac{1}{25} & \text{ if } (\be , A) \not\equiv 0 \mod 3.
  \end{cases}
\]
\end{lem}

\begin{proof}
By Appendix A,
\[
\begin{split}
  (\rho_\be u\,|\,e_A) & = \l(\om_{K_{12}} - \frac{1}9 \sum_{B\in \mathcal{A}} 
  \rho_\be e_B \,\Bigg|\, e_A\r) 
  = \frac{2}5 -\frac{1}9\left( (\rho_\be e_A\,|\,e_A)+\frac{80}{25}\right).
\end{split}
\]
Note that for any $A\in \mathcal{A}$, there exists exactly $80$ $\au$-invariant $\sqrt{2}A_2$-sublattices $B\in \mathcal{A}$
such that $A\neq B$ and $A\not\perp B$.  Thus,
\[
  (\rho_\be u\,|\,e_A) =
\begin{cases}
  \dfrac{2}{5} -\dfrac{1}{9}\left( \dfrac{2}{5} +\dfrac{80}{25}\right)=0 
  & \text{ if }  (\be , A) \equiv 0 \mod 3, 
  \vsb\\
  \dfrac{2}{5} -\dfrac{1}{9}\left( \dfrac{1}{25} +\dfrac{80}{25}\right)
  = \dfrac{1}{25} & \text{ if } (\be , A) \not\equiv 0 \mod 3,
\end{cases}
\]
as desired.
\end{proof}

\begin{lem}
Let $\be\in K_{12}$ be a vector of norm $4$, $6$, or $8$. Then we have
\[
  (\rho_\be u|u) =
  \begin{cases}
    \dfrac{1}{25} & \text{ if $\be$ has norm $4$ or $6$},
    \vsb\\
    0 & \text{ if $\be$ has norm $8$}.
  \end{cases}
\]
\end{lem}

\begin{proof}
By Appendix A, we have $u= \om_{K_{12}} -\frac{1}{9} \sum_{A\in \mathcal{A}} e_A$. 
Then
\[
\begin{split}
  (\rho_\be u\,|\,u) 
  & = \l( \om_{K_{12}} -\dfrac{1}{9}\sum_{A\in \mathcal{A}} \rho_\be e_A \,\Bigg|\, \om_{K_{12}} -\frac{1}9\sum_{B\in \mathcal{A}} e_B\r)
  \vsb\\
  &= 6 - 2\cd \frac{28}{5} + \frac{1}{81} \sum_{A,B\in \mathcal{A}} (\rho_\be e_A \,|\, e_B),
  \vsb\\
  & = -\frac{26}5 + \frac{1}{81} \sum_{A\in \mathcal{A}} \left( (\rho_\be e_A|e_A) + 80\cd \frac{1}{25}\right).
\end{split}
\]
Thus by the previous lemma, if $\be$ has norm $4$ or $6$, then
\[
\begin{split}
  (\rho_\be u\,|\,u) 
  &= -\frac{26}{5} + \frac{1}{81} \left( \frac{270}{6} \l( \frac{2}5 +\frac{80}{25}\r) 
  +\frac{756-270}{6}\l(\frac{1}{25}+\frac{80}{25}\r) \right) =\frac{1}{25}.
\end{split}
\]

If $\be $ has norm $8$, then
\[
  (\rho_\be u\,|\,u) = -\frac{26}{5} + \frac{1}{81}\left(\frac{216}{6} 
  \l(\frac{2}{5} +\frac{80}{25}\r) + \frac{756-216}{6}\l(\frac{1}{25}+\frac{80}{25}\r) 
  \right)= 0.
\]
Hence, we have the desired result.
\end{proof}

\begin{lem}\label{sigmaea}
Let $A, B$ be two distinct $\au$-invariant $\sqrt{2}A_2$-sublattices of 
$K_{12}$ such that $(A, B)\neq 0$. 
Then we have $\sigma_{e_A}( \rho_\be e_B) = \rho_{t_A\be  } ( e_{t_AB})$, 
where $t_A$ is the RSSD involution associated to $A$.
\end{lem}

\begin{proof}
First we note that  $t_At_B$ has order 3 and $A+B\cong A_2\otimes A_2$.

Let $g=t_At_B$. Then $B=gA$ and $gB=t_AB$. Suppose that
\[
  \ee^\al \cdot \hat{\au}(\ee^\al) = (-1)^i \hat{\au}( \ee^{-\al})
  \quad \text{ and }\quad
  \ee^{g\al}\cdot  \hat{\au}(\ee^{g\al}) = (-1)^j \hat{\au}( \ee^{-g\al})
\]
for any $\al \in A$.

Let $\epsilon\in \{\pm 1\}$ such that
$\ee^\al\cdot \ee^{g\al}= \epsilon \ee^{-g^2\al}$ and
$\ee^{-\al}\cdot \ee^{-g\al}= \epsilon \ee^{g^2\al}$.
Then
\[
\begin{split}
  \ee^{g^2\al} \cdot \hat{\au} \ee^{g^2\al}
  & =  \ee^{-\al}\cdot \ee^{-g\al} \cdot \hat{\au}(\ee^{-\al}\cdot \ee^{-g\al}),
  \\
  &= (-1) \ee^{-\al}\cdot\hat{\au} (\ee^{-\al})  \cdot \ee^{-g\al}\cdot
  \hat{\au}(\ee^{-g\al}),
  \\
  &= (-1)^{1+i+j}  \hat{\au}^2( \ee^{\al} \cdot \ee^{g\al}),
  \\
  &= (-1)^{1+i+j}  \epsilon \hat{\au}^2( \ee^{-g^2\al}).
\end{split}
\]
Note that $(g\al,  \au\al)=-1$ for any $\al\in A(4)$.

Since
\[
  e_A= \frac{2}5 \om_A + (-1)^i \frac{1}5 \sum_{\al\in A(4)} \ee^\al
  \quad \text{ and } \quad
  \rho_\be e_B  =\frac{2}5 \om_B + (-1)^j \frac{1}5 \sum_{\al\in A(4)}
  \rho_{\be} \ee^{g\al},
\]
we have
\[
\begin{split}
  (e_A)_{(1)}( \rho_\be e_B) =
  & \frac{1}{25}\left[4\times \frac{1}2 (\om_A+\om_B-\om_{t_AB})
  + (-1)^i \sum_{\al\in A(4)} \ee^\al\right.
  \\
  &\quad  +\left. (-1)^j \sum_{\al\in A(4)} \rho_{\be} \ee^{g\al}
  + (-1)^{i+j} \epsilon \sum_{\al\in A(4)}  \zeta^{(\be, g\al)} \ee^{-g^2\al}\right].
\end{split}
\]
Hence,
\[
\begin{split}
  \sigma_{e_A}( \rho_\be e_B) = & e_A+\rho_\be e_B -5 (e_A)_{(1)}( \rho_\be e_B)
  \\
  = & \frac{2}5 \om_{t_AB} + \frac{1}5 (-1)^{1+i+j}\epsilon \sum_{\al\in A(4)}
    \zeta^{(t_A\be,- g^2\al)}  \ee^{-g^2\al}
  \\
  =&  \rho_{t_A\be  } ( e_{t_AB})
\end{split}
\]
as desired.
\end{proof}

\medskip

Next we will compute $\sigma_u(\rho_\be u)$.

\begin{rem}\label{8uu}
When $\be$ has norm $8$, we have $\sigma_u(\rho_\be u) = \rho_\be u$ since
$(u|\rho_\be u)=0$.
\end{rem}

\begin{lem}\label{sigmaea2}
Let $\be \in K_{12}$ and $A\in \mathcal{A}$.
Then we have $\sigma_{e_A}(\rho_\be u) = \rho_{t_A\be} u$ if
$(A, \be )\not\equiv 0\mod 3$.
\end{lem}

\begin{proof}
Since $\rho_\be u= \om_{K_{12}} -\frac{1}9 \sum_{B\in \mathcal{A}} \rho_\be (e_B)$, we have
\[
  \sigma_{e_A} (\rho_\be(u)) = \om_{K_{12}} -\frac{1}9 \sum_{B\in \mathcal{A}} \sigma_{e_A}(\rho_\be (e_B)), 
  = \om_{K_{12}} -\frac{1}9 \sum_{B\in \mathcal{A}} (\rho_{t_A\be} (e_{t_AB})) 
  =\rho_{t_A\be } (u)
\]
as desired.
\end{proof}

\begin{lem}\label{norm4uu}
Let $\be \in K_{12}(4)$. Then we have
\[
\sigma_{u}(\rho_\be u) =
\rho_\be^{-1} e_{A(\be)} ,
\]
where $A(\be)=\Span_\Z\{\be, \au\be\}\cong \sqrt{2}A_2$.
\end{lem}

\begin{proof}
Let $\be$  be a norm $4$ vector in $K_{12}$. Then $A(\be)=\Span_\Z\{\be,  \au\be\}$ is a $\au$-invariant
$\sqrt{2}A_2$-sublattice.  By Lemma \ref{sigmaea2}, we have
\[
\sigma_{e_{A(\be)}} (\rho_\be  u ) = \rho_{t_{A(\be)} \be}(  u) =\rho_{\be}^{-1}  u.
\]
Notice that $t_{A(\be)} \be= -\be$ since $\be \in A(\be)$.  Therefore, we have
\[
\sigma_u( \rho_{\be}^{-1} e_{A(\be)})  = \rho_{\be}^{-1} \sigma_{\rho_{\be} u} e_{A(\be)}
 = \rho_{\be}^{-1} \sigma_{e_{A(\be)}} \rho_{\be}u = \rho_{\be}^{-1} \rho_{\be}^{-1} u  =\rho_{\be}  u.
 \]
Hence $\sigma_u (\rho_{\be}u) = \rho_{\be}^{-1} e_{A(\be)}$.
\end{proof}

\medskip

In the following, we will study the case when $\be$ has norm $6$.

\begin{nota}\label{Sa}
For any $\al\in K_{12}(4)$,  we denote $$S_\al=\{ \gamma\in K_{12}(4) \mid (\al, \gamma) =2\text{ but } \gamma \notin
A(\al)\}.$$ Notice that if $\gamma\in S_\al$, then $\al-\gamma$ is of norm $4$,  $\al-\gamma\in S_\al $ and $\al=\gamma+(\al
-\gamma)$.
\end{nota}

The next lemma can be obtained by a direct calculation.
\begin{lem}\label{norm60}
Let $\be$ be a norm $6$ vector   and  $\al$ a norm $4$ vector in $K_{12}$. Let  $S_\al$  be defined as in Notation \ref{Sa} and let
\[
Z_{\al, \be} =\{ \gamma\in S_\al\mid (\gamma, \beta ) \equiv 0 \mod 3\}.
\]
Then we have
\[
|Z_{\al, \be}|=
\begin{cases}
26 & \text{ if } (\al, \be )\equiv 0\mod 3,\\
30 & \text{ if } (\al, \be )\not\equiv 0\mod 3.
\end{cases}
\]
\end{lem}
\medskip

\begin{nota}
For any $\al \in K_{12}(4)$, we denote
\[
e_\al =
\begin{cases}
  \ee^\al & \text{ if }  \al \in S^1(4)\cup S^2(4),
  \\
  -\ee^\al &  \text{ if }  \al \in S^3(4)\cup W(4).
\end{cases}
\]
\end{nota}

\begin{lem}\label{norm6case}
  Let $\be$ be a norm $6$ vector in $K_{12}$.
  Then
  \[
    \l(\sum_{\al\in K_{12}(4)} e_\al\r)_{(1)}\l(\sum_{\al\in K_{12}(4)} \rho_\be e_\al\r)
    = 9\om_{K_{12}} - 3\sum_{\al\perp \be } e_\al -12 \sum_{\al\not\perp \be } \rho_\be^2 e_\al.
\]
\end{lem}

\begin{proof}
First we note that  for any $\gamma\in S_\al$, we have $e_\gamma e_{\al-\gamma}= e_{\al-\gamma}e_\gamma  $ and
$e_{-\au\al} e_{-\au^2 \al} = e_{-\au^2\al} e_{-\au \al} = - e_{\al}$ (cf. Lemma \ref{cocyle}).
Thus,
\[
\begin{split}
  & \l(\sum_{\al\in K_{12}(4)} e_\al\r)_{(1)}\l(\sum_{\al\in K_{12}(4)} \rho_\be e_\al\r)
  \\
  = &\ \sum_{\al\perp \be} \frac{\al(-1)^2\vac}2 
  + \sum_{\al\not\perp \be}(\zeta +\zeta^2) \frac{\al(-1)^2\vac}4
  \\
  &\ + \sum_{\al \perp \be }  \left( |Z_{\al, \be}| - \frac{ 80 -|Z_{\al, \be}|}2 -2\right) e_\al
  \ + \sum_{\al\not  \perp \be }  \left( |Z_{\al, \be}| - \frac{ 80 -|Z_{\al, \be}|}2 -2\right) \rho_{\be}^2e_\al\\
 = &\  9\om_{K_{12}} - 3\sum_{\al\perp \be } e_\al -12 \sum_{\al\not\perp \be } \rho_\be^2 e_\al
\end{split}
\]
by Lemma \ref{norm60}.  
Note also that $\frac{1}{2}\sum_{\al\perp \be} \al(-1)^2\vac =90 \om_{K_{12}}$ and 
$\frac{1}{4}\sum_{\al\not\perp \be}\al(-1)^2\vac = 81 \om_{K_{12}}$ since there are exactly $45$ $\au$-invariant $\sqrt{2}A_2$ -sublattices perpendicular to
$\be$ and $81$ $\au$-invariant $\sqrt{2}A_2$ -sublattices  not perpendicular to $\be$ (see Lemma \ref{orthal} ).
\end{proof}

\begin{lem}\label{uu}
  Let $\be$ be a norm $6$ vector in $K_{12}$. 
  Then we have $\sigma_{u}(\rho_\be u) = \rho_\be^{-1} u$.
\end{lem}

\begin{proof}
Suppose $\be$ has norm $6$. Then by Lemma \ref{norm6case}, we have
\[
  \l(\sum_{\al\in K_{12}(4)} e_\al\r)_{(1)}\l(\sum_{\al\in K_{12}(4)} \rho_\be e_\al\r)
  = 9\om_{K_{12}} - 3\sum_{\al\perp \be } e_\al -12 \sum_{\al\not\perp \be } \rho_\be^2 e_\al.
\]
Hence,
\[
\begin{split}
u_1(\rho_\be u) = &\left( \frac{1}{15}\om_{K_{12}}+\frac{1}{45} \sum_{\al\in K_{12}} e_\al\right )_{(1)} \left( \frac{1}{15}\om_{K_{12}} +\frac{1}{45} \sum_{\al\in K_{12}}\rho_\be e_\al\right ),\\
 =& \frac{1}{45^2} \left( 18 \om_{K_{12}} +6 \sum_{\al\in K_{12}} e_\al + 6\sum_{\al\in K_{12}} \rho_\be e_\al + 9\om_{K_{12}} - 3\sum_{\al\perp \be } e_\al -12 \sum_{\al\not\perp \be } \rho_\be^2 e_\al\right)
 \\
 =& \frac{1}{5}\left( \frac{1}{15}\om_{K_{12}} + \frac{1}{45} \sum_{\al\perp \be } e_\al - \frac{2}{45}  \sum_{\al\not\perp \be } \rho_\be^2 e_\al\right) .
\end{split}
\]
Therefore,
\[
  \sigma_u(\rho_\be u) =u+\rho_\be u-5u_1(\rho_\be u) = \rho_\be^2 u
\]
as desired.
\end{proof}

\begin{thm}
Let $E=E_1\cup E_2$ and $G=\la \sigma_e \mid e\in E\ra< \aut(V_{K_{12}}^{\hat{\au}})$ . Then we have $G\cong
{^+\Omega^-}(8,3)$.
\end{thm}

\begin{proof}
Let $X\cong \mathbb{F}_3^8$ be a nondegenerate quadratic space of $(-)$-type. Write $X=U\oplus Y$, where $U$ is a hyperbolic
plane and $Y$ is a 6-dimensional non-degenerate quadratic space of $(-)$-type. We shall identify $Y$ with $K_{12}/(1-\au)
K_{12}$.

Note that $U$ has exactly two isotropic lines, say $\la v_0\ra$ and $\la v_0'\ra$, two non-singular vectors $v_1$, $-v_1$ of norm $1$ and two non-singular vectors $v_2$, $-v_2$ of norm $2$ (or $-1$).

Now let $x= v+y\in U\oplus Y$ be a non-singular vector of norm 1 and let $\la x\ra= \Span_{\mathbb{F}_3}\{ x\}$ be the line spanned by $x$.

If $v=0$, then $y$ has norm $1$ and it is represented by a norm $4$ vector $\be \in K_{12}$, i.e., $y= \be +(1-\au)K_{12}\in
K_{12}/(1-\au) K_{12}$. In this case, we assign $\la x\ra $ to $e_{A(\be)}$, where $A(\be)= \Span_{\Z}\{\be , \au\be\}\cong
\sqrt{2}A_2$.

If $v=v_0$ is isotropic, then $y$ is again represented by a norm $4$ vector $\be \in K_{12}$. We assign $\la x\ra $ to $\rho_\be e_{A(\be)}$.

If $v=v_0'$, then $y$ is represented by a norm $4$ vector $\be \in K_{12}$ and we assign $\la x\ra $ to $\rho_\be^{-1} u $.

If $v=v_1$, then $y$ is isotropic and is represented by a norm $6$ vector $\gamma$. In this case, we assign  $\la x\ra $ to $\rho_\gamma u $.

If $v=v_2$, then $y$ has norm $-1$ and is represented by a norm $8$ vector $\delta$. In this case, we  assign  $\la x\ra $ to $\rho_\delta u $.

Then, by Remark \ref{8uu} and  Lemmas \ref{norm4uu} and \ref{uu}, the assignment above defines an ${^+\Omega^-}(8,3)$-map
\[
  \varphi: \left\{ 1\text{-spaces generated}\atop \text{by norm 1 vectors}\right \} 
  \longto E.
\]
It induces a group homomorphism
\[
\begin{array}{cccc}
  \tilde{\varphi}: &{^+\Omega^-}(8,3) &\longto &G
  \\
  &r_x &\longmapsto& \sigma_{\varphi(x)}.
\end{array}
\]
Recall that $G$ contains a subgroup $\bar{H}=\la \sigma_e\mid e\in E_1\ra  \cong 3^6{:} {^+\Omega^-}(6,3)$, which is a maximal
subgroup of ${^+\Omega^-}(8,3)$ and $\sigma_u $ is non-trivial.  Therefore, we must have  $G\cong {^+\Omega^-}(8,3)$.

It is also easy to show that  the assignment $\varphi$ is one-to-one and onto. Hence the group homomorphism $\tilde{\varphi}$ is in
fact an isomorphism.
\end{proof}

\appendix
\section{An explicit definition for $c=4/5$ Virasoro vectors in $V_{K_{12}}^{\hat{\au}}$ }
Next we will give an explicit definition for $c=4/5$ Virasoro vectors in $V_{K_{12}}^{\hat{\au}}$. First let  $\hat{\au}$ be a lift of
$\au$ in $V_{K_{12}}$.

\begin{lem}\label{Kv}
Let $\mathcal{A}$ be the set of all $\au$-invariant $\sqrt{2}A_2$ sublattices 
in $K_{12}$ and let $e_A, A\in \mathcal{A}$ be defined as in Lemma \ref{eA}. 
Then the element
\[
  \tilde{u} =\frac{1}{9} \sum_{A\in \mathcal{A}} e_A
\]
is a Virasoro vector of central charge $56/5$ in $V_{K_{12}}^{\hat{\au}}$. 
Therefore, $\om_{K_{12}}- \tilde{u}$ is a Virasoro vector of central charge $4/5$ 
in $V_{K_{12}}^{\hat{\au}}$.
\end{lem}

\begin{proof}
First we note that there are $126$ $\au$-invariant $\sqrt{2}A_2$-sublattices 
in $K_{12}$. 
By Theorem \ref{thm:3.4}, for any $A$, $B\in \mathcal{A}$, we have
\[
  (e_A)_{(1)} e_B =
  \begin{cases}
    2e_A& \text{ if } A=B, 
    \vsb\\
    0& \text{ if } A\perp B,
    \vsb\\
    \dfrac{1}{5}( e_A + e_B -\sigma_{e_A}e_B) & \text{ otherwise.}
  \end{cases}
\]
Note also that for any $A\in \mathcal{A}$, there exists exactly $80$ $\au$-invariant $\sqrt{2}A_2$-sublattices $B\in \mathcal{A}$
such that $A\neq B$ and $A\not\perp B$.

Therefore,
\[
\begin{split}
  \tilde{u}_{(1)}\tilde{u} &= \frac{1}{81} \l(\sum_{A\in \mathcal{A}}e_A\r)_{(1)}
  \l(\sum_{B\in \mathcal{A}}e_B\r)
  = \frac{1}{81}\sum_{A\in \mathcal{A}}\left( 2e_A + (e_A)_{(1)}\l(\sum_{B\neq A} e_B\r)\right)
  \\
  &= \frac{1}{81}\sum_{A\in \mathcal{A}}\left( 2e_A + \sum_{B\not\perp A} \frac{1}{5}(e_A+e_B- \sigma_{e_A}e_B)\right)
  = \frac{1}{81}\sum_{A\in \mathcal{A}}\l( 2e_A+\frac{80}{5}e_A\r)
  \\
  &= \frac{2}{9} \sum_{A\in \mathcal{A}} e_A =2\tilde{u},
\end{split}
\]
and
\[
\begin{array}{ll}
  (\tilde{u}|\tilde{u}) 
  &= \dfrac{1}{81} \l(\dsum_{A\in \mathcal{A}}e_A \,\Bigg|\, \dsum_{B\in \mathcal{A}}e_B\r)
  = \dfrac{1}{81}\dsum_{A\in \mathcal{A}}\left ( (e_A | e_A) + \dsum_{B\neq A}(e_A | e_B) \right)
  \vsb\\
  &= \dfrac{1}{81}\cdot 126 \cdot \l(\dfrac{2}5 +\dfrac{1}{25} \cdot 80\r) =\dfrac{28}{5}.
\end{array}
\]
Since $e_A$ is fixed by $\hat{\au}$ for any $A\in \mathcal{A}$, both $\tilde{u}$ and $\om_{K_{12}}- \tilde{u}$ are fixed by
$\hat{\au}$.
\end{proof}

\begin{rem}
Note that the definition of $e_A$ depends on  the choice of the lift $\hat{\au}$. Hence the definition of $u$ also depends on  the
choice of the lift $\hat{\au}$
\end{rem}

\medskip

Next we will show that the Virasoro vector $u$ defined in Remark 
\ref{defu} agrees with $\om_{K_{12}}- \tilde{u}$. 
Let $M$, $N$, $F=M\cap N$ and $J=\ann_{M+N}(F)\cong K_{12}$ be defined as 
in Notation \ref{Q} and let $\hat{\au}=\tau_{e_M}\tau_{e_N}$.

As in the proof of Lemma \ref{inK12}, set $S^1=J\cap M$, $S^2= J\cap N$ 
and $S^3=\au S^2$. 
Denote $T^1= M\setminus (F\perp S^1)$, $T^2=M\setminus (F\perp S^2)$, 
$T^3=\au T^2$ and $W=J\setminus (S^1+S^2)$. 
Then
\[
\begin{split}
  e_M &= \frac{1}{16}(\om_F+ \om_{S^1}) +\frac{1}{32}\left(
  \sum_{\al\in F(4)} \ee^\al + \sum_{\al\in S^1(4)} \ee^\al 
  +\sum_{\al\in T^1(4)} \ee^\al\right),
  \\
  e_N&= \frac{1}{16}(\om_F+ \om_{S^2}) +\frac{1}{32}\left(
  \sum_{\al\in F(4)} \ee^\al + \sum_{\al\in S^2(4)} \ee^\al 
  +\sum_{\al\in T^2(4)} \ee^\al\right).
\end{split}
\]
By direct calculation, we have
\[
  (\om_F+ \om_{S^1})_{(1)}(\om_F+ \om_{S^2})
  = 2\om_F+ \frac{1}{2}(\om_{S^1} +\om_{S^2}-\om_{S^3}),
\]
\[
  (\om_F+ \om_{S^1})_{(1)} \left(\sum_{\al\in F(4)} \ee^\al + \sum_{\al\in S^2(4)} 
  \ee^\al +\sum_{\al\in T^2(4)} \ee^\al\right)
  = 2 \sum_{\al\in F(4)} \ee^\al + \frac{1}{2} \sum_{\al\in S^2(4)} \ee^\al 
  + \sum_{\al\in T^2(4)} \ee^\al,
\]
\[
  (\om_F+ \om_{S^2})_{(1)} \left(
  \sum_{\al\in F(4)} \ee^\al + \sum_{\al\in S^1(4)} \ee^\al 
  +\sum_{\al\in T^1(4)} \ee^\al\right)
  = 2 \sum_{\al\in F(4)} \ee^\al + \frac{1}{2} \sum_{\al\in S^1(4)} \ee^\al 
  + \sum_{\al\in T^1(4)} \ee^\al,
\]
and
\[
\begin{split}
&\ \left(
  \sum_{\al\in F(4)} \ee^\al + \sum_{\al\in S^1(4)} \ee^\al 
  +\sum_{\al\in T^1(4)} \ee^\al\right)_{(1)}
  \left(\sum_{\al\in F(4)} \ee^\al + \sum_{\al\in S^2(4)} \ee^\al 
  +\sum_{\al\in T^2(4)} \ee^\al\right)
  \\
  =& 12\om_F + 2 \sum_{\al\in F(4)} \ee^\al +2\sum_{\al\in T^1(4)} \ee^\al 
  +2 \sum_{\al\in T^2(4)} \ee^\al + \sum_{\al\in S^3(4)} \ee^\al 
  + 2\sum_{\al\in T^3(4)} \ee^\al + 3\sum_{\al\in W(4)} \ee^\al.
\end{split}
\]
Hence,
\[
\begin{split}
  u & =\frac{2^6}{135} (2e_M +2e_N + \hat{\au}(e_N) -16(e_M)_{(1)}(e_N))
  \\
  &=\frac{1}{45} \left( 2(\om_{S^1}+\om_{S^2}+\om_{S^3})
  +\sum_{\al\in S^1(4)} \ee^\al+\sum_{\al\in S^2(4)} \ee^\al 
  -\sum_{\al\in S^3(4)} \ee^\al  -\sum_{\al\in W(4)} \ee^\al\right)
  \\
  &= \frac{1}{15} \om_{K_{12}} + \frac{1}{45}\left( \sum_{\al\in S^1(4)} \ee^\al
  +\sum_{\al\in S^2(4)} \ee^\al -\sum_{\al\in S^3(4)} \ee^\al  
  -\sum_{\al\in W(4)} \ee^\al\right) .
\end{split}
\]


\begin{lem}\label{cocyle}
Let $\al, \be$ be norm $4$ vectors.
\\
(1) Suppose that $\al, \be \in S^3(4)$ and $(\al, \be )=-2$. 
Then  we have $\ee^\al \ee^\be = -\ee^{\al+\be}$.
\\
(2) Suppose $\al \in S^1(4)$. 
Then $\hat{\au}(\ee^\al) = \ee^{\au\al}$ and $\hat{\au}^2(\ee^\al) = -\ee^{\au^2\al}$.
\\
(3) Let $\al\in W(4)$. Then  $\varepsilon(\al , \au\al) \equiv 0 \mod 2$.
\\
(4) Let $\al, \be\in W(4)$ such that $\al+\be\in W(4)$. 
Then $\varepsilon(\al, \be)\equiv 1 \mod 2 $ unless $\be \in \Span_\Z\{ \al, \au\al\}$.
\end{lem}

\begin{proof}
(1):~Let $\al, \be \in S^3(4)$. 
Then there exist $\al', \be'\in S^1(4)$ such that $\al =\al'+\au\al'$, 
$\be=\be'+\au\be'$ and $(\al,\be)=(\al',\be')$.  
In this case, we have
\[
  \ee^\al \ee^\be 
  = \ee^ {\al'} \ee^{\au\al'} \ee^{\be'}\ee^{\au\be'}
  = (-1)^{ (\au\al', \be') } \ee^ {\al'} \ee^{\be'} \ee^{\au\al'}\ee^{\au\be'}
  = - \ee^{\al'+\be' +\au(\al'+\be')} =- \ee^{\al+\be}.
\]
(2):~First we note that $\ee^\al= \frac{1}{2}\left( (\ee^\al+ \ee^{-\au^2\al}) 
+ (\ee^{\al}- \ee^{-\au^2\al})\right)$. 
For $\alpha \in S^1(4)$,  we have
\[
(e_N)_{(1)} (\ee^\al+ \ee^{-\au^2\al}) = \frac{1}{16} (\ee^\al+ \ee^{-\au^2\al})\quad  \text{ and } \quad
(e_N)_{(1)} (\ee^\al- \ee^{-\au^2\al})=0.
\]
Recall that 
$$
  e_N= \frac{1}{16} \om_N + \frac{1}{32} \sum_{\be \in N(4)} \ee^\be .
$$ 
Moreover, $(\be,\al) =-2$ if and only if $\be=\au \al$, and $(\be , -\au^2 \al)=-2$ 
if and only if $\be=-g\al$. 
Thus, we have
\[
\begin{split}
  \tau_{e_N}\ee^\al
  &= \tau_{e_N} \l(\frac{1}{2}\left( (\ee^\al+ \ee^{-\au^2\al}) + (\ee^{\al}- \ee^{-\au^2\al})\right)\r)
  \\
  &= \frac{1}{2}\left( -(\ee^\al+ \ee^{-\au^2\al}) + (\ee^{\al}- \ee^{-\au^2\al})\right)
  \\
  &= - \ee^{-\au^2\al}.
\end{split}
\]

By the same argument, we also have
\[
  \tau_{e_M}\ee^{-\au^2\al} = - \ee^{\au\al}.
\]
Hence, $\hat{\au} \ee^\al = \tau_{e_M}\tau_{e_N} \ee^\al = -\tau_{e_M} \ee^{\au^2\al}
=\ee^{\au\al}$ and
\[
  \hat{\au}^2 \ee^\al 
  = \tau_{e_N}\tau_{e_M} \ee^\al 
  = \tau_{e_N} \ee^{-\al}
  =  - \ee^{\au^2\al}.
\]
as desired.
\\
(3):~Let $\al\in W(4)$. 
Then $\al=x+\au y$ for some $x,y\in M(4)$, where $x=a+b$ and $y =-a-b'$ with 
$a\in (M\cap N)^*$, $b,b'\in (S^1)^*$ and $b-b'\in S^1$. 
Then $\au \al= \au x+\au^2 y = a+\au b -a+\au^2 b'= (a+b')+ (- a+\au(b+b'))$ and thus
\[
  \varepsilon(\al,\au\al) 
  =(\au y, a+b') 
  = (-a-\au b', a+b')
  = -\frac{4}3+\frac{1}2\cdot \frac{8}3
  = 0.
\]
(4):~Let $\al=x+\au y $ and $\be =x'+\au y'$ such that
$$
  x=a+b,~~~
  y= -a- b',~~~
  x'=c+d,~~~ 
  y'=-c-d',
$$
and $(b,b')=(d, d')=-4/3$.
Then $\al+\be \in W(4)$ implies $(\al, \be)=-2$ and
\[
  (b,d) +(b',d') + \frac{1}{2}  ((b,d') +(b',d)) =-2.
\]
Since $x,x'y,y'\in M$, which is doubly even, we  have $(x,x')$, $(y, y') \in 2\Z$ 
and thus
\[
  (b,d) \equiv (b',d') \equiv -(a,c)  \mod 2.
\]
If we further assume $\varepsilon(\al, \be) \equiv (x, \au y') \equiv (x', \au y) 
\equiv 0 \mod 2$, then we also have
\[
  \frac{1}{2} (b, d') \equiv (a,c) \equiv \frac{1}{2} (b',d)  \mod 2.
\]
Note that $(a,c) \equiv \pm 2/3 \mod 2$ and $|(a, c)|\leq 4/3$.

\medskip

\textbf{Case 1.}~Suppose $(a,c) \equiv 2/3\mod 2 $. 
Then $a+c$ has integral norm and hence $a+c\in M\cap N$. 
Therefore, $x+x'$ and $y+y'$ are contained in $M\cap N +S^1$ and thus 
$\al +\be = x+x' + \au (y+y') \in S^1+S^2$, which contradicts our 
assumption that $\al+\be \in W$.

\medskip

\textbf{Case 2.}~Suppose $(a,c) \equiv -2/3\mod 2$. 
Then $(b,d')/2 = 4/3 $ or $-2/3$.
If $(b,d')/2 = 4/3$, then $(b, d') =8/3$ and hence $b=d'$. 
Thus
\[
  (d',d) +(b', b) + \frac{1}{2}  \l( \frac{8}{3} +(b',d)\r) =-2
\]
and we have $(b', d) = -4/3$. 
Note that $(b,b')=(d, d')=-4/3$. 
Therefore, $(d, b+b')= -8/3$ and $d= -( b+b')$.
Hence, $\be = -(b+b') -\au b = \au^2 \al$.

Similarly, if $(b',d)/2 = 4/3$, we also have $\be \in \Span_\Z\{ \al, \au\al\}$.

\medskip

Now suppose $(b,d')/2= (b',d)/2 = -2/3$. 
Then we have
\[
  (b,d) +(b', d')  - \frac{4}{3} 
  = -2 \quad  \text{ and  } \quad (b,d) +(b', d') 
  = -\frac{2}{3}.
\]
Since $(b, d)=(b',d') \equiv -(a,c) \equiv 2/3 \mod 2$, we have
\[
  (b,d)= -\frac{4}{3},~~ (b', d') =\frac{2}{3}
  \quad  \text{ or } \quad 
  (b,d)= \frac{2}{3},~~ (b', d') =-\frac{4}{3}.
\]
Without loss of generality, we assume $(b,d)= -4/3$ and $(b', d')=2/3$. 
Then
\[
  (d,b+b') = -\frac{8}{3} \quad \text{  and } \quad (b,d+d') =-\frac{8}{3}
\]
and hence $-d= b+b'$ and $-b= d+d'$. 
This implies $b'=d'= -(b+d)$.  
This is a contradiction since $(b', d')=2/3$.
\end{proof}

\begin{prop}
We have $u=\om_{K_{12}} - \frac{1}{9}\sum_{A\in \mathcal{A}} e_A $.
\end{prop}

\begin{proof}

First we note that $\mathcal{A}$ is a union of 21 disjoint $\sqrt{2}A_2$-frames, 
i.e., sets of 6 mutually orthogonal $\au$-invariant$\sqrt{2}A_2$-sublattices. 
Therefore, we have $\om_{K_{12}}=(1/21)\sum_{A\in \mathcal{A}} \om_A$.
Moreover, by Lemma \ref{cocyle}, we have
\[
\begin{split}
  e_{A(\al)} 
  = &\frac{2}{5} \om_{A(\al)} - \frac{1}{5} \left(\ee^{\al}+ \ee^{-\al} 
  + \hat{\au}(\ee^{\al} +\ee^{-\al})+ \hat{\au}^2( \ee^{\al} + \ee^{-\al})\right)
  \\
  = &\frac{2}{5} \om_{A(\al)} -\frac{1}{5} \left(\ee^{\al}+ \ee^{-\al} 
  + (\ee^{\au\al} +\ee^{-\au\al}) - ( \ee^{\au^2\al} + \ee^{-\au^2\al})\right)
\end{split}
\]
if $\al\in S^1$ and
\[
  e_{A(\al)} =  \frac{2}{5} \om_{A(\al)} + \frac{1}{5} \left(\ee^{\al}
  + \ee^{-\al} + (\ee^{\au\al} +\ee^{-\au\al}) + ( \ee^{\au^2\al} 
  + \ee^{-\au^2\al})\right)
\]
if $\al\in W$.
Hence, we have $u =\om_{K_{12}} -\frac{1}{9}\sum_{A\in \mathcal{A}} e_A$ as desired.
\end{proof}

\small

\end{document}